\documentclass{article}

\usepackage{amsmath,amssymb,amsthm,setspace,dsfont,hyperref,tikz}
\usetikzlibrary{matrix,arrows,decorations.pathmorphing}

\usepackage{kpfonts}
\usepackage[T1]{fontenc}
\hypersetup{
    colorlinks,
    citecolor=black,
    filecolor=black,
    linkcolor=black,
    urlcolor=black
}

\renewcommand{\thispagestyle}[1]{} 

\usepackage{fancyhdr} \usepackage{lastpage}

\newtheorem{theorem}{Theorem}[section]
\newtheorem{corollary}[theorem]{Corollary}
\newtheorem{lemma}[theorem]{Lemma}
\newtheorem{proposition}[theorem]{Proposition}

\theoremstyle{definition}
 
\theoremstyle{remark}
\newtheorem{remark}[theorem]{Remark} 

\theoremstyle{example}

\numberwithin{equation}{section} 

\newcommand{\comment}[1]{}

\newcommand{\R}{\mathbb R}
\newcommand{\N}{\mathbb N}

\newcommand{\eps}{\varepsilon}

\newcommand{\ls}{\leqslant}
\newcommand{\gr}{\geqslant}

\providecommand{\abs}[1]{\lvert#1\rvert}
\providecommand{\norm}[1]{\lVert#1\rVert}

\providecommand{\conv}[1]{\mathop{\rm conv}\left\{#1\right\}}

\providecommand{\dim}[1]{\mathop{\rm dim}(#1)}
\providecommand{\det}[1]{\mathop{\rm det}\left(#1\right)}

\allowdisplaybreaks

\newcommand{\fg}{\mathfrak{g}}
\newcommand{\fk}{\mathfrak{k}}
\newcommand{\fs}{\mathfrak{s}}
\newcommand{\fa}{\mathfrak{a}}
\newcommand{\fm}{\mathfrak{m}}
\newcommand{\fr}{\mathfrak{r}}
\newcommand{\fn}{\mathfrak{n}}
\newcommand{\fp}{\mathfrak{p}}
\newcommand{\fz}{\mathfrak{z}}
\newcommand{\fgl}{\mathfrak{gl}}
\renewcommand{\:}{\, : \,}

\def\Ad{\mathrm{Ad}}
\def\ad{\mathrm{ad}}
\def\Tr{\mathrm{Tr}}

\def\SL{\mathrm{SL}}
\def\SO{\mathrm{SO}}
\def\O{\mathrm{O}}
\def\GL{\mathrm{GL}}
\def\S{\mathrm{S}}

\def\I{\mathrm{I}}
\def\M{\mathrm{M}}

\begin{document}
\large 

\title{Bounding marginal densities via affine isoperimetry}

\author{Susanna Dann \and Grigoris Paouris \and Peter Pivovarov}

\maketitle

\begin{abstract}
Let $\mu$ be a probability measure on $\R^n$ with a bounded density
$f$.  We prove that the marginals of $f$ on most subspaces are
well-bounded. For product measures, studied recently by Rudelson and
Vershynin, our results show there is a trade-off between the strength
of such bounds and the probability with which they hold.  Our proof
rests on new affinely-invariant extremal inequalities for certain
averages of $f$ on the Grassmannian and affine Grassmannian. These are
motivated by Lutwak's dual affine quermassintegrals for convex sets.
We show that key invariance properties of the latter, due to Grinberg,
extend to families of functions.  The inequalities we obtain can be
viewed as functional analogues of results due to Busemann--Straus,
Grinberg and Schneider. As an application, we show that without any
additional assumptions on $\mu$, any marginal $\pi_E(\mu)$, or a small
perturbation thereof, satisfies a nearly optimal small-ball
probability.

\end{abstract}

\section{Introduction}

In this paper, we discuss connections between affine isoperimetric
inequalities in convex geometry and concentration results for high
dimensional probability distributions.  We address the following
question: if $\mu$ is a probability measure on $\R^n$ with a bounded
density, to what extent are its marginal densities also bounded?
Recall that if $\mu$ has density $f$ and $E$ is a $k$-dimensional
subspace of $\R^n$, the density of the marginal $\pi_E(\mu)$ on $E$ is
given by \begin{equation}
\label{eqn:marginal}
f_{\pi_E(\mu)}(x) = \int_{E^{\perp}+x} f(y)dy \quad (x\in E).
\end{equation}

Rudelson and Vershynin \cite{RV_marginals} recently proved that if
$f(x)=\prod_{i=1}^nf_i(x_i)$, where each $f_i$ is a bounded density on
$\R$, then for every $1\ls k\ls n$ and every $k$-dimensional subspace
$E$,
\begin{equation}
  \label{eqn:RV_bound}
  \norm{f_{\pi_E(\mu)}}_{\infty}^{1/k} \ls C\max_{i\ls n}\norm{f_i}_{\infty}
\end{equation}
where $C$ is a numeric constant and $\norm{\cdot}_{\infty}$ is the
$L^{\infty}$-norm. On the other hand, even for products, the stronger
inequality
\begin{equation}
  \label{eqn:RV_bound_comp}
  \norm{f_{\pi_E(\mu)}}_{\infty}^{1/k} \ls C\norm{f}_{\infty}^{1/n}
\end{equation}
need not hold for {\it all} subspaces $E$; indeed, if $f_1,\ldots,f_k$
are very peaked, the left-hand side of (\ref{eqn:RV_bound_comp}) can
be arbitrarily large for the coordinate subspace $E=\mathop{\rm
  span}\{e_1,\ldots,e_k\}$, while the right-hand side can be
well-behaved since 
\begin{equation}
  \label{eqn:means}
  \norm{f}_{\infty}^{1/n}\ls \sqrt[n]{\norm{f_1}_{\infty}\cdots
    \norm{f_n}_{\infty}}\ls \max_{i\ls n}\norm{f_i}_{\infty}.
\end{equation}
Nevertheless, we show that for an arbitrary
bounded density $f$, most of its marginals nearly satisfy
(\ref{eqn:RV_bound_comp}), which we quantify with respect to the Haar
probability measure $\mu_{n,k}$ on the Grassmannian manifold $G_{n,k}$
of all $k$-dimensional subspaces of $\R^n$.

\begin{theorem} 
  \label{thm:marginal_intro}
  Let $\mu$ be a probability measure on $\R^n$ with a bounded density
  $f$. Then for each $1\ls k \ls n-1$, there exists
  ${\cal{A}}\subseteq G_{n,k}$ with $\mu_{n,k}( {\cal{A}}) \gr 1-
  2e^{-kn}$ such that for every $E\in {\cal{A}}$,
    \begin{equation}
      \label{eqn:marginal_intro_1}
      f_{\pi_E(\mu)}(x)^{1/k}\ls C\norm{f}_{\infty}^{1/n}
    \end{equation} 
    for all $x \in E$, except possibly on a set of
    $\pi_E(\mu)$-measure less than $e^{-kn}$.
\end{theorem}

Thus given $f$, first sampling $E\in G_{n,k}$ according to $\mu_{n,k}$
and then $x\in E$ according to $\pi_E(\mu)$,
(\ref{eqn:marginal_intro_1}) holds with overwhelming probability; on
$G_{n,k}$ this is optimal (see Lemma
\ref{lemma:Gaussian_sharp}). Furthermore, one must exclude exceptional
sets of positive $\pi_E(\mu)$-measure as can be seen by considering a
neighborhood of a Besicovitch set (see, e.g., \cite[Chapter 9]{BP}).

As discussed in \cite{RV_marginals}, bounds for marginals are
connected to small-ball probabilities, which are useful in random
matrix theory e.g., \cite{RV_ICM}. If $X$ is a random vector in $\R^n$
with density $f$, then $f_{\pi_E(\mu)}$ gives the density of the
orthogonal projection $P_EX$ of $X$ onto $E$. When $f$ is bounded,
Theorem \ref{thm:marginal_intro} implies that for every $E\in
{\cal{A}}$, $\varepsilon>0$ and any $z\in E$,
\begin{equation}
  \label{eqn:marginal_intro_2}
  \mathbb{P}\left(\abs{P_EX-z}\ls \varepsilon \sqrt{k} \right) \ls
  (C_1\varepsilon \norm{f}_{\infty}^{1/n})^{kn/(n+1)},
\end{equation}
where $\abs{\cdot}$ denotes the Euclidean norm.  In contrast,
(\ref{eqn:RV_bound}) implies that for any $z\in E$
\begin{equation}
  \label{eqn:smallball}
  \mathbb{P}\left(\abs{P_EX-z}\ls \varepsilon \sqrt{k} \right) \ls
  (C_2\varepsilon \max_{i}\norm{f_i}_{\infty})^{k}.
\end{equation}
Thus if $f=\prod_i f_i$ and $\norm{f_i}_{\infty}$ are not identical,
the base in the probability in (\ref{eqn:marginal_intro_2}) is smaller
than that in (\ref{eqn:smallball}) (cf. (\ref{eqn:means})) but the
exponent $\frac{n}{n+1}$ is slightly worse.  Furthermore, for
subspaces $E$ that do not belong to $\mathcal{A}$, one can perturb
them to ensure that (\ref{eqn:marginal_intro_2}) holds (see Corollary
\ref{result-corollary}).

The goal of this paper is to show that a purely probabilistic
statement such as Theorem \ref{thm:marginal_intro} is ultimately based
on an affine-invariance property of certain integrals on the
Grassmannian $G_{n,k}$ and affine Grassmannian $M_{n,k}$ and
corresponding extremal inequalities. Here $M_{n,k}$ is equipped with
its cannonical rigid-motion invariant measure $\nu_{n,k}$ (see Section
\ref{section:prelim}). In particular, for non-negative, bounded
integrable functions $f$ on $\R^n$, we consider
\begin{equation}
  \label{eqn:ratioGnk}
    \int_{G_{n,k}}
    \frac{\left(\int_{E}f(x)dx\right)^{n}}{\norm{f\lvert_E}_{\infty}^{n-k}}
    d\mu_{n,k}(E)
\end{equation} where $f\vert_E$ is the restriction of $f$ to $E$, and 
\begin{equation}
  \label{eqn:ratioMnk}
    \int_{M_{n,k}} \frac{\left(\int_{F}f(x)dx\right)^{n+1}}
        {\norm{f\vert_{F}}_{\infty}^{n-k}}d\nu_{n,k}(F).
\end{equation} 

Our interest in such quantities stems from the following notion: for
$1\ls k< n$, the {\it dual affine quermassintegrals} of a convex body
$K\subset \R^n$ are defined by
\begin{equation}
  \widetilde{\Phi}_{n-k}(K) =
  \frac{\omega_n}{\omega_k}\left(\int_{G_{n,k}}\abs{K\cap
    E}^n d\mu_{n,k}(E)\right)^{1/n}
\end{equation}
where $\omega_n$ denotes the volume of the Euclidean ball $B_2^n$ in
$\R^n$ of radius one and $\abs{\cdot}$ denotes Lebesgue measure. These
were introduced by Lutwak (see \cite{Lutwak_isepiphanic},
\cite{Lutwak_intersection} for background) and have proved to be an
indespensable tool for quantitative questions concerning
high-dimensional probability distributions, e.g.,
\cite{E_Milman_mixed}, \cite{PaoPiv_smallball},
\cite{PaourisValettas}. In \cite{Grinberg}, Grinberg proved that
$\widetilde{\Phi}_{n-k}(K)=\widetilde{\Phi}_{n-k}(SK)$ for each
volume-pre\-serv\-ing linear transformation $S$.  Motivated by
Grinberg's result, we prove that the quantities in
(\ref{eqn:ratioGnk}) and (\ref{eqn:ratioMnk}) are also invariant under
volume preserving linear and affine transformations, respectively.
Our argument uses the structure of semi-simple Lie groups.

In the case when $f=\mathds{1}_K$, where $K$ is a convex body (or
compact set), both (\ref{eqn:ratioGnk}) and (\ref{eqn:ratioMnk})
satisfy corresponding affine isoperimetric inequalities. In particular, a
result of Busemann-Straus \cite{BusemannStraus} and Grinberg
\cite{Grinberg} states that if $K$ is a convex body in $\R^n$ and
$1\ls k\ls n-1$, then
\begin{equation}
  \label{eqn:Grinberg}
  \int_{G_{n,k}}\abs{K\cap E}^{n}d\mu_{n,k}(E) \ls
  \frac{\omega_{k}^{n}}{\omega_{n}^k}\abs{K}^{k};
\end{equation}
when $k>1$, equality holds only for origin-symmetric ellipsoids.  The
$k=n-1$ case is Busemann's seminal intersection inequality
\cite{Busemann_volume}. For the other endpoint, i.e. when $k=1$,
(\ref{eqn:Grinberg}) is an equality for $\mathds{1}_K$, evident from
expressing the integral in spherical coordinates.

For the affine Grassmannian $M_{n,k}$, an inequality of Schneider
\cite{Schneider_flats} states that if $K$ is a convex body in $\R^n$ and
$1\ls k \ls n-1$, then
\begin{equation}
  \label{eqn:Schneider}
  \int_{M_{n,k}}\abs{K\cap F}^{n+1}d\nu_{n,k}(F) \ls
  \frac{\omega_k^{n+1}\omega_{n(k+1)}}{\omega_n^{k+1}\omega_{k(n+1)}}\abs{K}^{k+1};
\end{equation} 
when $k>1$, equality holds if and only if $K$ is an $n$-dimensional
ellipsoid; when $k=1$, equality holds if and only if $K$ is a convex
body, which follows from the classical Crofton formula (e.g.,
\cite[Theorem 5.1.1]{SchneiderWeil}).

While many of the latter inequalities also hold for non-convex sets, the
equality cases require additional care.  Gardner \cite{Gardner_dual}
generalized (\ref{eqn:Grinberg}) and (\ref{eqn:Schneider}), among
other related inequalities, to the class of bounded, Borel measurable
sets with a precise characterization of equality cases, making use of
results due to Pfiefer \cite{Pfiefer}, \cite{Pfiefer90}.  In this
paper, we extend such inequalities to bounded integrable
functions. The analysis of equality cases in the functional setting
rests heavily on their results.

\begin{theorem}
  \label{thm:Busemann}
  Let $1\ls k\ls n-1$ and let $f$ be a non-negative, bounded
  integrable function on $\R^n$.  Then
  \begin{equation}
    \label{eqn:thm:Busemann}
    \int_{G_{n,k}}
    \frac{\left(\int_{E}f(x)dx\right)^{n}}{\norm{f\lvert_E}_{\infty}^{n-k}}
    d\mu_{n,k}(E) \ls \frac{\omega_{k}^n}{\omega_n^{k}}
    \left(\int_{\R^n}f(x)dx\right)^k.
  \end{equation} 
\end{theorem}

We also discuss the equality cases in the latter theorem under a mild
assumption on $f$ in which case equality holds in
(\ref{eqn:thm:Busemann}) when $k>1$ if and only if $f =
a\mathds{1}_{\mathcal{E}}$ a.e., where $\mathcal{E}$ is an
origin-symmetric ellipsoid and $a$ is a positive
constant. Furthermore, we prove a more genaral statement for $q\ls k$
different functions, as well as different powers (see Section
\ref{section:main_results}).

The corresponding result on $M_{n,k}$ is the following inequality.

\begin{theorem}
  \label{thm:Schneider}
  Let $1\ls k\ls n-1$ and let $f$ be non-negative, bounded integrable
  function on $\R^n$.  Then
  \begin{equation}
    \label{eqn:thm:Schneider}
    \int_{M_{n,k}} \frac{\left(\int_{F}f(x)dx\right)^{n+1}}
        {\norm{f\vert_{F}}_{\infty}^{n-k}}d\nu_{n,k}(F) \ls
        \frac{\omega_k^{n+1}\omega_{n(k+1)}}{\omega_n^{k+1}\omega_{k(n+1)}}
        \left(\int_{\R^n}f(x)dx\right)^{k+1}.
  \end{equation}
\end{theorem}

Under a mild assumption on $f$, we also prove that equality holds in
(\ref{eqn:thm:Schneider}) when $k>1$ if and only if $f =
a\mathds{1}_{\mathcal{E}}$ a.e., where $\mathcal{E}$ is an ellipsoid
and $a$ is a positive constant.


One can interpret Theorem \ref{thm:Schneider} as an inequality about
the $k$-plane transform.  Recall that the $k$-plane transform
$T_{n,k}$ applied to a function $f$ on $\R^n$ is defined by
\begin{equation}
  T_{n,k}(f)(F) = \int_F f(x) dx \quad (F\in M_{n,k}).
\end{equation}
When $k=n-1$, $T_{n,k}$ is the Radon transform and when $k=1$, it is
the X-ray transform.  The $k$-plane transform satisfies several key
inequalites.  In particular, for each $q\in [1,n+1]$, there is a
unique $p\in [1,(n+1)/(k+1)]$ such that
\begin{equation}
  \label{eqn:Christ_pq}
  \norm{T_{n,k}(f)}_q \ls C(n,k,q) \norm{f}_p
\end{equation} 
for all $f\in L_p$. The latter is a special case of a result due to
Christ \cite{Christ_kplane}, extending work by Drury
\cite{Drury_kplane}; see also the article of Baernstein and Loss
\cite{BaernsteinLoss_kplane} for related work and a conjecture about
the extremal functions; for recent research in this direction, see
Christ \cite{Christ_Radon}, Druout \cite{Druout_kplane} and Flock
\cite{Flock_kplane} and the references therein.  The endpoint
inequality $q=n+1$ and $p=(n+1)/(k+1)$ in (\ref{eqn:Christ_pq}) also
satisfies an affine-invariance property \cite{Christ_Radon},
\cite{Druout_kplane}.

{\bf Organization: } We close the introduction with a few words on the
main tools that we use and the organization of the paper.  Section
\ref{section:prelim} is reserved for notation and background results,
including formulas from integral geometry such as the
Blaschke-Petkantschin formulas.  In Section \ref{section:affine}, we
treat affine invariance using the structure of semi-simple Lie groups
which we then specialize to the Grassmannian and affine
Grassmannian. In Section \ref{section:functional}, we recall a
functional version of Busemann's random simplex inequality
\cite{Busemann_volume}, and its variant due to Groemer
\cite{Groemer_simplex}, \cite{Groemer_polytope}, from
\cite{PaoPiv_probtake}; the latter makes essential use of Christ's
form \cite{Christ_kplane} of the Rogers-Brascamp-Lieb-Luttinger
inequality \cite{Rogers_single}, \cite{BLL}.  The ratios in
(\ref{eqn:ratioGnk}) and (\ref{eqn:ratioMnk}) arise naturally in a
suitable normalized form of the main inequality in
\cite{PaoPiv_probtake}.  In Section \ref{section:main_results}, we
prove Theorems \ref{thm:Busemann} and \ref{thm:Schneider}. We finish
the paper in Section \ref{section:marginals} with a more general
version of Theorem \ref{thm:marginal_intro} and we discuss connections
to the Hyperplane Conjecture from convex geometry.


\section{Preliminaries}
\label{section:prelim}

The setting is $\R^n$ with the canonical inner-product $\langle \cdot,
\cdot \rangle$, Euclidean norm $\abs{\cdot}$ and standard unit vector
basis $e_1,\ldots,e_{n}$.  We also use $\abs{\cdot}$ for Lebesgue
measure and the absolute value of a scalar, the use of which will be
clear from the context. The Euclidean ball of radius one is $B_2^n$
with volume $\omega_n=\abs{B_2^n}$.  We reserve $D_n$ for the
Euclidean ball of volume one, i.e., $D_n=r_n B_2^n$, where
$r_n=\omega_n^{-1/n}$.  The unit sphere is $S^{n-1}$ and is equipped
with the Haar probability measure $\sigma$. As mentioned, the Haar
probability measure on the Grassmannian $G_{n,k}$ is denoted by
$\mu_{n,k}$.  The affine Grassmannian $M_{n,k}$ is equipped with a
measure as follows: for $A\subset M_{n,k}$,
\begin{equation}
  \nu_{n,k}(A) = \int_{G_{n,k}} \abs{\{x\in E^{\perp}: x+E \in
    A\}}d\mu_{n,k}(E).
\end{equation}  Henceforth,
we will write simply $dF$ rather than $d\nu_{n,k}(F)$ for integrals
over $M_{n,k}$; similarly, $dE$ instead of $d\mu_{n,k}(E)$
for integrals on $G_{n,k}$. Note that $\mu_{n,k}$ is a probability
measure while $\nu_{n,k}$ is normalized so that $\nu_{n,k}(\{F\in
M_{n,k}:F\cap B_2^n\not=\emptyset\}) = \omega_{n-k}$. We use
$c_1,c_2,C,\ldots$ etc for positive numeric constants.

We will make use of the following integral geometric identities, often
referred to as the Blaschke-Petkantschin formulas; see e.g.,
\cite[Chapter 7.2]{SchneiderWeil}, \cite{Drury_kplane},\cite[Lemmas 5.1 \&
  5.5]{Gardner_dual}, \cite{Miles}; see also the generalization given
in \cite[Appendix A]{E_Milman_intersection}.

\begin{theorem} 
  \label{thm:BPGnk}
  Let $1\ls q\ls k\ls n$. Suppose that $G$ is a non-negative, Borel
  measurable function on $(\R^n)^q$.  Then
  \begin{eqnarray*}
    \lefteqn{\int_{(\R^n)^q} G(x_1,\ldots,x_q)dx_1\ldots dx_q} \\ & &
    = c_{n,k,q} \int_{G_{n,k}}\int_{E^q} G(x_1,\ldots,x_q)
    \abs{\mathop{\rm conv}\{0,x_1,\ldots,x_q\}}^{n-k} dx_1\ldots dx_q dE,
  \end{eqnarray*}
where 
\begin{equation}
  \label{eqn:cnkq}
  c_{n,k,q} = (q!)^{n-k} \frac{\omega_{n-q+1}\cdots
    \omega_n}{\omega_{k-q+1}\cdots \omega_{k}}.
\end{equation}
\end{theorem}
\pagebreak
\begin{theorem}
  \label{thm:BPMnk}
  Let $1\ls q\ls k\ls n$. Suppose that $G$ is a non-negative Borel
  function on $(\R^n)^{q+1}$. Then
  \begin{eqnarray*}
    \lefteqn{\int_{(\R^n)^{q+1}} G(x_1,\ldots,x_{q+1})dx_1\ldots dx_{q+1}} \\
    & & =c_{n,k,q}
    \int_{M_{n,k}}\int_{F^{q+1}} G(x_1,\ldots,x_{q+1})
    \abs{\mathop{\rm conv}\{x_1,\ldots,x_{q+1}\}}^{n-k} dx_1\ldots dx_{q+1} dF,
  \end{eqnarray*}where $c_{n,k,q}$ is defined in (\ref{eqn:cnkq}).
\end{theorem}

If $A\subset \R^n$ is a Borel set with finite volume, the symmetric
rearrangement $A^{\ast }$ of $A$ is the (open) Euclidean ball centered
at the origin whose volume is equal to that of $A$.  The symmetric
decreasing rearrangement of $1_A$ is defined by
$(1_A)^{\ast}:=1_{A^{\ast }}$. If $f:{\R}^n\rightarrow {\R}^+$ is an
integrable function, its symmetric decreasing rearrangement $f^{\ast}$
is defined by
\begin{equation*}
  f^{\ast }(x)=\int_0^{\infty }1^{\ast }_{\{ f> t\}}(x)dt
  =\int_0^{\infty }1_{\{ f>t\}^{\ast }}(x)dt.
\end{equation*}
The latter can be compared with the ``layer-cake representation''
of $f$:
\begin{equation}
  \label{eqn:layer_cake}
  f(x)=\int_0^{\infty }1_{\{ f> t\}}(x)dt;
\end{equation}
see \cite[Theorem 1.13]{LL_book}.  The function $f^{\ast}$ is
radially-symmetric, decreasing and equimeasurable with $f$, i.e.,
$\{f>\alpha\}$ and $\{f^*>\alpha\}$ have the same volume for each
$\alpha > 0$.  By equimeasurability one has $\norm{f}_p=\norm{f^*}_p$
for each $1\ls p\ls \infty$, where $\norm{\cdot}_p$ denotes the
$L_p(\R^n)$-norm.  We refer the reader to the book \cite{LL_book} for
further background material on rearrangements of functions.


\section{Affine invariance}
\label{section:affine}

In this section we discuss linear and affine invariance properties of
the quantities in (\ref{eqn:ratioGnk}) and (\ref{eqn:ratioMnk}),
respectively, as well as generalizations. We start with the former
and prove the following theorem.

\begin{theorem}
\label{th_aigr}
Let $m$ be a positive integer and let $p_i, \alpha_i$, for $i=1, \dots, m$ be real numbers.
Let $f_i$ be bounded functions on $\R^n$, $f_i \in L^{p_i}(\R^n)$.
Define 
$$ I(f_1, \dots, f_m) := \int\limits_{G_{n,k}} \prod\limits_{i=1}^m \| f_i|_x \|^{\alpha_i}_{p_i} \,\, dx \, .$$ 
Whenever this quantity is finite and $\sum_{i=1}^m \frac{\alpha_i}{p_i} =n$, for any volume-preserving linear transformation $g$, we have
$$ I(g\cdot f_1, \dots, g\cdot f_m) = I(f_1, \dots, f_m) \, ,$$
where $g\cdot f_i(t)=f_i(g^{-1}t)$.
\end{theorem}

\begin{remark}
Letting $p_i, \alpha_i, q_i, \beta_i$ satisfy $\sum_{i=1}^m \left(\frac{\alpha_i}{p_i} - \frac{\beta_j}{q_j}\right) =n$, Theorem \ref{th_aigr} yields the linear invariance of 
$$ \int\limits_{G_{n,k}} \prod\limits_{i=1}^m  \frac{ \| f_i|_x \|^{\alpha_i}_{p_i} }{\| f_i|_x \|^{\beta_i}_{q_i}}\,\, dx \, ,$$
and letting $q_i \rightarrow \infty$, also of
$$ \int\limits_{G_{n,k}} \prod\limits_{i=1}^m  \frac{ \| f_i|_x \|^{\alpha_i}_{p_i} }{\| f_i|_x \|^{\beta_i}_{\infty}}\,\, dx \, .$$
Note that in the latter case there are no restriction on the $\beta_i$'s.
\end{remark}

Grinberg's approach \cite{Grinberg}, which in turn draws on
Furstenberg-Tzkoni \cite{FT}, can be adapted to our setting, although
we prefer to give a more self-contained proof using the structure of
semi-simple Lie groups. For this reason, the notation in this section
differs somewhat from the rest of the paper.

\subsection{Semi-simple Lie groups}
We recall some basic facts from the theory of semi-simple Lie groups
as needed for our later discussion about the Grassmannian manifold. We
follow the presentation from \cite{OlafssonPasquale_cos}. Further
information and details about this topic can be found for example in
\cite{Knapp}.

Let $G$ be a non-compact connected semi-simple Lie group with a finite center. We denote its Lie algebra by $\fg$. An involution $\theta:G \to G$ is called a Cartan involution if $K=G^{\theta}:=\{a\in G \: \theta(a)=a \}$ is a maximal compact subgroup of $G$. In this case $K$ is connected. We fix a Cartan involution $\theta$ on $G$ and the corresponding maximal compact subgroup $K$. The derived involution $\dot{\theta} \: \fg \to \fg$ will also be denoted by $\theta$. We have $\fg = \fk \oplus \fs$, where $\fk = \fg^{\theta}$ is the Lie algebra of $K$ and $\fs = \{X\in \fg \: \theta(X)=-X \}$. Moreover, $ [\fk, \fk] \subset \fk, [\fs, \fs] \subset \fk, [\fk, \fs] \subset \fs$. By $\Ad$ and $\ad$, we denote the adjoint representation of the Lie group $G$ and of the Lie algebra $\fg$, respectively. The Killing form on $\fg$ is given by $\left\langle X, Y \right\rangle := \Tr(\ad(X) \ad(Y))$. And the product 
$$ (X,Y) := - \left\langle X, \theta(Y) \right\rangle $$ 
is an inner product on $\fg$. Note that $\ad(X)^{\ast}=-\ad(\theta(X))$. In particular, for $X \in \fs$, $\ad(X)$ is a symmetric operator and hence diagonalizable over the reals. Let $\fa \subset \fs$ be abelian. Then $\ad \, \fa$ is a family of commuting symmetric transformations and thus can be diagonalized simultaneously, with real eigenvalues. For each linear functional $\lambda$ on $\fa$, $\lambda \in \fa^{\ast}$, let 
$$ \fg_{\lambda}= \{X\in \fg \: [H,X] = \lambda(H) X \text{ for all } H\in \fa \} $$
and set 
$$ \fm = \{X\in\fg \: [H,X]=0 \text{ for all } H\in \fa \text{ and } X \perp \fa \}\,.$$
Then $\fg_0=\fm \oplus \fa$. Let $\Delta = \{\lambda \in \fa^{\ast}\setminus\{0\} \: \fg_{\lambda}\neq \{0\} \}$. Elements in $\Delta$ are called restricted roots. We have
$$ \fg = \fm \oplus \fa \oplus \bigoplus_{\lambda \in \Delta} \fg_{\lambda} \, .$$
For $\lambda, \mu \in \Delta \cup \{0\}$, we have $[\fg_{\lambda}, \fg_{\mu}]\subset \fg_{\lambda+\mu}$. In particular, $[\fm \oplus \fa, \fg_{\lambda}]\subset \fg_{\lambda}$ for all $\lambda \in \Delta$. Further, let $\fa_{\fr} = \{H\in \fa \: \lambda(H)\neq 0 \text{ for all } \lambda \in \Delta \}$. Fix $H\in \fa_{\fr}$ and let $\Delta^{+} = \{ \lambda \in \Delta \: \lambda(H) > 0 \}$. Elements in $\Delta^{+}$ are called positive roots. We have $\Delta = \Delta^{+} \cup - \Delta^{+}$, $\Delta^{+} \cap - \Delta^{+} = \emptyset$ and $(\Delta^{+} + \Delta^{+}) \cap \Delta \subset \Delta^{+}$. It follows that 
$$ \fn := \bigoplus_{\lambda\in \Delta^{+}} \fg_{\lambda} $$
is a nilpotent subalgebra of $\fg$ normalized by the Lie algebra $\fp := \fm \oplus \fa \oplus \fn$. In fact, $\fp$ is a parabolic subalgebra of $\fg$. It is maximal if $\dim \fa = 1$. 

Let $P:=\{g\in G \: \Ad(g)\fp \subset \fp \}$. Then $P$ is a closed subgroup of $G$ with the Lie algebra $\fp$. Let $A:=\exp \fa$ and $N:=\exp \fn$ be analytic subgroups of $G$ with Lie algebras $\fa$ and $\fn$, respectively. The groups $A,N$ are closed. $A$ is abelian and $N$ is nilpotent. Denote by $M_{o}$ the analytic subgroup of $G$ with the Lie algebra $\fm$. Let $M:=Z_K(A)M_{o}$, where $Z_K(A)$ stands for the centralizer of $A$ in $K$. $M$ is a closed subgroup of $G$ with finitely many connected components. The map $M \times A \times N \ni (m,a,n)\mapsto man \in P$ is an analytic diffeomorphism. We have $G = KP$. Further, let $L := K \cap M$ and $\widetilde{X}:=K/L$, then
$$ K \cap P = L \text{ and } \widetilde{X}=G/P \, .$$
The group $G$ acts on $\widetilde{X}$. Write $$ G \ni g=k(g)m(g)a(g)n(g)\quad ((k(g), m(g), a(g), n(g)) \in K \times M \times A \times N).$$ The map $g \mapsto (a(g),n(g))$ is analytic and the map $g\mapsto a(g)$ is right $MN$-invariant. Thus we can view $a(\cdot)$ as a map $G/MN \to A$. The elements $k(g), m(g)$ are not uniquely defined. However, the map
$$ g\mapsto k(g)L \in \widetilde{X} $$
is well-defined and analytic. Set $x_o:=eL$, where $e$ denotes the identity element of $G$. The action of $G$ on $\widetilde{X}$ can now be described by $g \cdot (k \cdot x_o)=k(gk)\cdot x_o$. For $x=k\cdot x_o \in \widetilde{X}$ and $g\in G$, we set $a(gx):=a(gk)$. 

We normalize the invariant measure on $\widetilde{X}$ to have total mass one. For $f\in C(\widetilde{X})$,
$$ \int_{\widetilde{X}} f(x) dx = \int_K f(k\cdot x_0) dk \, .$$

For $\lambda \in \fa^{\ast}$ define a homomorphism $\chi_{\lambda} \: P\to \R$ by
$$ \chi_{\lambda}(m \, \exp(H) \, n)=e^{\lambda(H)}, \text{ where } m\in M, H\in \fa, n\in N \, .$$ 
The shorthand notation for $\chi_{\lambda}(p)$ is $p^{\lambda}$. For $\lambda \in \Delta$, set $m_{\lambda}:=\dim \fg_{\lambda}$ and $\rho := \frac{1}{2} \sum_{\lambda \in \Delta^+} m_{\lambda} \lambda \in \fa^{\ast} $. Note that $p^{2 \rho} = |\det \Ad(p)|_{\fn}|$.

We shall need the following well known lemma. 
\begin{lemma}\label{lemma_antcvf}
Let $f\in L^1(\widetilde{X})$ and $g\in G$. Then
$$ \int_{\widetilde{X}} f(g\cdot x) a(gx)^{-2\rho} dx = \int_{\widetilde{X}} f(x) dx \, .$$
\end{lemma}   

\begin{proof}
See Lemma 5.19 on p. 197 in \cite{Helgason}. This result is formulated there in a slightly different form. However, the precise equality appears in the proof as equation (25).
\end{proof}
 
\subsection{The Grassmannian manifold}
Now we apply the general structure theory of semi-simple Lie groups discussed above to the special case of Grassmann manifolds. 

Let $G_{n,k}$ denote the Grassmann manifold of all oriented $k$-di\-men\-sion\-al subspaces of $\R^n$ and set $r=n-k$. Note that $G_{n,k} \cong G_{n,r}$. 

Set $G=\SL(n)$, then $\fg$ is the set of $n\times n$ matrices with trace zero. The homomorphism $\theta \: G \to G \: x \to x^{-tr}$ is a Cartan involution on $G$ with $K=G^{\theta}=\SO(n)$. The corresponding Cartan involution on $\fg$ is $\theta(X)=-X^{tr}$. Denote by $\M(n)$ the set of $n\times n$ matrices. We have
\begin{align*}
\fk & = \{X\in M(n) \: X^{tr}=-X \text{ and } \Tr(X)=0 \} \, , \\
\fs & = \{X\in M(n) \: X^{tr}=X \text{ and } \Tr(X)=0 \} \, .
\end{align*}
The Killing form on $\fg$ is given by $\left\langle X,Y \right\rangle = 2n \Tr(XY) $. 
For $l\in \N$, denote by $\I_l$ the $l\times l$ identity matrix. Let 
$$ H_o = \begin{pmatrix} \frac{r}{n} \I_k & 0 \\ 0 & -\frac{k}{n} \I_r  \end{pmatrix} \in \fs \, .$$
We define $\fa := \R H_o$, then $\fm=\{X\in \fz_{\fg}(\fa) \: \left\langle X,H_o \right\rangle = 0 \}$. Fix $\alpha \in \fa^{\ast}$ so that $\alpha(H_o)=1$.  We choose $\Delta = \{ \alpha, -\alpha \}$ and $\Delta^+=\{ \alpha \}$. We have
\begin{align*}
\fm \oplus \fa 
	& = \left\{\begin{pmatrix} X & 0 \\ 0 & Y  \end{pmatrix} \: 
			\begin{array}{l}	X\in M(k),  \\	Y\in M(r)  \end{array}
      \text{ and } \Tr X + \Tr Y = 0 \right\} \\
	& = \fs( \fgl(k) \times \fgl(r)) \, , \\
\fn & = \left\{\begin{pmatrix} 0 & X \\ 0 & 0  \end{pmatrix}  \: X\in M(k\times r) \right\} \, , \\
MA &= \left\{m(a,b):= \begin{pmatrix} a & 0 \\ 0 & b  \end{pmatrix} \:
			\begin{array}{l}	a\in \GL(k),  \\	b\in \GL(r)  \end{array}
			\text{ and } \det a \det b =1 \right\} \\
	&= \S(\GL(k) \times \GL(r)) \, , 
\end{align*}
\begin{align*}
A &= \left\{ \begin{pmatrix} s \I_k & 0 \\ 0 & t \I_r  \end{pmatrix} \: s,t>0 \text{ with } s^k t^r=1 \right\} \, , \\
M &= \left\{ \begin{pmatrix} a & 0 \\ 0 & b  \end{pmatrix} \: \det a, \det b = \pm 1 \text{ and } \det a \det b =1 \right\} \, , \\
N &= \left\{n(X):= \begin{pmatrix} \I_k & X \\ 0 & \I_r  \end{pmatrix} \: X\in M(k\times r) \right\} \, , \\
P &= \left\{\begin{pmatrix} a & X \\ 0 & b  \end{pmatrix} \: 
			\begin{array}{l}	a\in \GL(k),  \\	b\in \GL(r),  \end{array}
			\begin{array}{l}  \det a \det b =1,  \\	X\in M(k\times r)  \end{array} \right\} \, , \\
L &= \S(\O(k)\times \O(r)) \, .
\end{align*}
Let $e_1, \dots, e_n$ be the canonical basis for $\R^n$. Set $x_o=\R e_1 \oplus \cdots \oplus \R e_k \in G_{n,k}$. We have $G_{n,k} = K \cdot x_o \cong K/L = G/P$. 

We identify $\fa^{\ast}$ with $\R$ via $\lambda \mapsto \frac{n}{kr} \lambda(H_o)$. The inverse of this map is $z \mapsto z \frac{kr}{n} \alpha$. Since $\dim \fn = kr$, we have
\begin{equation}\label{eq_rsorirn}
\fa^{\ast} \ni \rho = \frac{kr}{2} \alpha \longleftrightarrow \frac{n}{2} \in \R \,.
\end{equation}
For $z\in \R$, we write $p^z$ instead of $p^{z\frac{kr}{n} \alpha}$.

\subsection{Linear invariance for functions on $G_{n,k}$}
For $g\in \SL(n)$ and $x\in G_{n,k}$ denote by $J_g(x)$ the Jacobian
determinant of the transformation $x\mapsto gx$. Then for $f\in
L^1(x)$,
\begin{equation}\label{eq_cvip}
\int_{x} f(t) \, |J_g(x)| \, dt = \int_{gx} f(g^{-1}t) dt \, .
\end{equation}

\begin{lemma}\label{lemma_json}
For $k\in \SO(n)$ and $x\in G_{n,k}$, we have $|J_k(x)|=1$.
\end{lemma}

\begin{proof}
Let $K\subset \R^n$ be measurable and $f=\mathds{1}_{K}$ be the characteristic function of $K$. For $g\in \SL(n)$, we compute
$$ \int_{gx} \mathds{1}_{K}(g^{-1}t) dt = \int_{gx} \mathds{1}_{gK}(t) dt = |gK \cap gx | = |g(K\cap x)| \, . $$
If $g\in \SO(n)$, then $|g(K\cap x)| = |(K\cap x)| = \int_x \mathds{1}_{K}(t) dt $. And the claim follows for characteristic functions.

By an analogous computation, the claim follows for simple functions. For a general function $f$, the claim follows by approximating $f$ from below by simple functions.
\end{proof}

Recall the following multiplicative property of the Jacobian: Let $T:Z\to Y$ and $S:Y \to X$, then $S\circ T: Z\to X$ and 
$$ J_{S\circ T}(z) = J_S(T(z)) \,\, J_T(z) \, \text{ with } z\in Z.$$

\begin{lemma}\label{lemma_jfgx}
For $z\in \R, x\in G_{n,k}$ and $g\in \SL(n)$, we have $ |J_g(x)|^{z} = a(gx)^{z} $.
\end{lemma}

\begin{proof}
Write $g=k p$ with $k\in K$ and $ p \in P$. By the multiplicative property of the Jacobian, we have
$$ |J_g(x_o)| = |J_{k p}(x_o)| = |J_k(p \cdot x_o) J_p(x_o)| = |J_p(x_o)| \, ,$$
where the last equality follows by Lemma \ref{lemma_json}.

Let $y\in \R^k$, then $\tilde{y}=(y, 0, \dots, 0) \in x_o$. Decompose $p = m(a,b)n(X)$ and compute
$$ p \cdot \tilde{y}  
		= \begin{pmatrix} a & 0 \\ 0 & b  \end{pmatrix} \begin{pmatrix} \I_k & X \\ 0 & \I_r  \end{pmatrix} \begin{pmatrix} y  \\ 0   \end{pmatrix}
	 = \begin{pmatrix} a & X \\ 0 & b  \end{pmatrix} \begin{pmatrix} y  \\ 0   \end{pmatrix} 
	 = \begin{pmatrix} a y  \\ 0   \end{pmatrix} \, . $$
Thus $J_p(x_o) = \det a$.

For some $m_o \in M$ and $s\in \R$, $m(a,b)=m_o \exp(s H_o)$. Let $m_o=\begin{pmatrix} u & 0 \\ 0 & v  \end{pmatrix}$ and note that $\exp(s H_o) = \begin{pmatrix} e^{s\frac{r}{n}} \I_k & 0 \\ 0 & e^{-s\frac{k}{n}} \I_r  \end{pmatrix}$. It follows that $a=e^{s\frac{r}{n}} u$. Hence $\det a = e^{s\frac{kr}{n}} \det u$ and so $|\det a| = e^{s\frac{kr}{n}}$. 

Next observe: $ \exp(s H_o)^{z} = \chi_{z\frac{kr}{n}\alpha}(\exp(s H_o)) = e^{z\frac{kr}{n}\alpha(s H_o)} = e^{s \frac{kr}{n} z}$. Thus we have shown
$$ |J_g(x_o)|^{z} = |J_p(x_o)|^z = |\det a|^z = e^{s \frac{kr}{n} z} = \exp(s H_o)^{z} = a(g)^{z} \, .$$
\noindent
Write a general $ x\in G_{n,k}$ as $x=k \cdot x_o$. We have $|J_g(x)|=|J_g(k \cdot x_o)|=|J_g(k \cdot x_o) J_k(x_o)|= |J_{gk}(x_o)|$. Hence
$$ |J_g(x)|^{z} = |J_{gk}(x_o)|^{z} =a(gk)^{z}=a(gx)^{z} \,.$$
\end{proof}

Substituting the result of Lemma \ref{lemma_jfgx} into (\ref{eq_cvip}) yields  
\begin{equation}\label{eq_ecvfip}
a(gx) \int_{x} f(t) dt = \int_{gx} f(g^{-1}t) dt \, .
\end{equation}

\begin{proof}[Proof of Theorem \ref{th_aigr}]
Fix $m\in \N$ and let $f_i, p_i, \alpha_i, g$ be as described in the statement of the theorem. We compute
\begin{align*}
I(g\cdot f_1, \dots, g\cdot f_m) 
	&= \int\limits_{G_{n,k}} \prod\limits_{i=1}^m \left( \int_x | f_i(g^{-1} t) |^{p_i} \, dt \right)^{\frac{\alpha_i}{p_i} }  \,\, dx  \\
	&= \int\limits_{G_{n,k}} a(gx)^{-n} \prod\limits_{i=1}^m \left( \int_{gx} | f_i(g^{-1} t) |^{p_i} \, dt \right)^{\frac{\alpha_i}{p_i} } \,\, dx \\
	&= \int\limits_{G_{n,k}} a(gx)^{-n} \prod\limits_{i=1}^m \left(a(gx) \int_{x} | f_i(t) |^{p_i} \, dt \right)^{\frac{\alpha_i}{p_i} } \,\, dx \\
	&= \int\limits_{G_{n,k}} a(gx)^{-n} a(gx)^{\sum \frac{\alpha_i}{p_i}} \prod\limits_{i=1}^m \| f_i|_x \|^{\alpha_i}_{p_i}  \,\, dx \\
	&= I(f_1, \dots, f_m) \, .
\end{align*}
The second equality follows by Lemma \ref{lemma_antcvf} and  (\ref{eq_rsorirn}). The third equality follows by (\ref{eq_ecvfip}).
\end{proof}

\subsection{Affine invariance for functions on $M_{n,k}$}

The linear-invariance property from the previous section can be transferred to an affine-invariance property on the affine Grassmannian.

\begin{theorem}\label{th_aiagr}
Let $m$ be a positive integer and let $p_i, \alpha_i$, for $i=1, \dots, m$ be real numbers. Let $f_i$ be bounded functions on $\R^n$, $f_i \in L^{p_i}(\R^n)$. Define 
$$ \tilde{I}(f_1, \dots, f_m) := \int\limits_{M_{n,k}} \prod\limits_{i=1}^m \| f_i|_x \|^{\alpha_i}_{p_i} \,\, dx \, .$$ 
Whenever this quantity is finite and $\sum_{i=1}^m \frac{\alpha_i}{p_i} =n+1$, for any volume-preserving affine transformation $g$, we have
$$ \tilde{I}(g\cdot f_1, \dots, g\cdot f_m) = \tilde{I}(f_1, \dots, f_m) \, ,$$
where $g\cdot f_i(t)=f_i(g^{-1}t)$.
\end{theorem}

To prove this theorem we will need an analog of Lemma \ref{lemma_antcvf}, for which, in turn, we need a couple of simple observations.

\begin{lemma}\label{lemma_jfoc}
For $x\in G_{n,k}$ and $g\in \SL(n)$, we have $ |J_g(x^{\perp})| = |J_g(x)|^{-1} $.
\end{lemma}

\begin{proof}
\begin{align*}
	\int_{\R^n} g \cdot f(z) dz 
	&= \int_{\R^n} f(z) dz = \int_{x^{\perp}} \int_x f(t+s) dt ds \\
	&= \int_{x^{\perp}} \int_{gx} f(g^{-1} t+s) \,\, |J_g(x)|^{-1} dt ds \\
	&= \int_{gx^{\perp}} \int_{gx} f(g^{-1}t+g^{-1}s) \,\, |J_g(x)|^{-1} dt \,\, |J_g(x^{\perp})|^{-1} ds \\
	&= |J_g(x)|^{-1} \,\, |J_g(x^{\perp})|^{-1} \int_{\R^n} g \cdot f(z) dz \, ,
\end{align*}
where we applied $(3.2)$ twice.
\end{proof}

\begin{lemma}\label{lemma_jfas}
Let $g\in \SL(n)$. For $y\in M_{n,k}$, let $x\in G_{n,k}$ and $s\in \R^n$ be so that $y=x+s$. Then $ |J_g(y)| = |J_g(x)| $.
\end{lemma}

\begin{proof}
\begin{align*}
	\int_{gy} f(g^{-1} z) \,\, |J_g(y)|^{-1} dz &= \int_{x+s} f(z) dz = \int_{x} f(t+s) dt \\
	&= \int_{gx} f(g^{-1}t+s) \,\, |J_g(x)|^{-1} dt \\
	&= \int_{g(x+s)} f(g^{-1}z) \,\, |J_g(x)|^{-1} dz \\
	&= \int_{gy} f(g^{-1}z) \,\, |J_g(x)|^{-1} dz \, ,
\end{align*}
where we again employed $(3.2)$.
\end{proof}

\begin{lemma}\label{lemma_cvoag}
Let $g\in \SL(n)$ and $f \in L^1(M_{n,k})$. Then 
$$ \int_{M_{n,k}} f(gy)  \,\, |J_g(y)|^{-(n+1)} dy = \int_{M_{n,k}} f(y)  \,\, dy. $$
\end{lemma}

\begin{proof}
\begin{align*}
	\int_{M_{n,k}} f(gy) dy 
	&= \int_{G_{n,k}} \int_{x^{\perp}} f(g(x+s)) ds \,\, dx \\
	&= \int_{G_{n,k}} \int_{x^{\perp}} f(x+gs) ds  \,\, |J_{g^{-1}}(x)|^{-n} dx \\
	&= \int_{G_{n,k}} \int_{gx^{\perp}} f(x+s)  \,\, |J_g(x^{\perp})|^{-1} ds  \,\, |J_g(x)|^n dx \\
	&= \int_{G_{n,k}} \int_{gx^{\perp}} f(x+s) ds  \,\, |J_g(x)|^{n+1} dx \\
	&= \int_{M_{n,k}} f(y)  \,\, |J_g(y)|^{n+1} dy  \, ,
\end{align*}
where the second equality follows by Lemma \ref{lemma_antcvf} along with Lemma \ref{lemma_jfgx}. 
\end{proof}

\begin{proof}[Proof of Theorem \ref{th_aiagr}]
The result follows by an analogous computation to the one in the proof of Theorem \ref{th_aigr}.
\end{proof}

For related affine invariance properties, see \cite{Christ_Radon} and \cite{Druout_kplane}.


\section{Functional forms of isoperimetric inequalities}

\label{section:functional}

We start by recalling the main result from \cite{PaoPiv_probtake}.
For positive integers $k, n$ and vectors $x_1,\ldots,x_k$ in $\R^n$,
we view the $k\times n$ matrix $[x_1\cdots x_k]$ as an operator from
$\R^k$ to $\R^n$. If $C\subset\R^k$, then
\begin{equation}
  [x_1\cdots x_k]C = \left\{\sum_{i=1}^k c_i x_i: c=(c_i)\in
  C\right\}.
\end{equation}
For example, if $C=\conv{0,e_1,\ldots,e_k}$, then
\begin{equation}
  \label{eqn:conv0}
[x_1\cdots x_k]\conv{0,e_1,\ldots,e_k} = \conv{0,x_1,\ldots,x_k}.
\end{equation}
Similarly, if $x_1,\ldots,x_{k+1}\in \R^n$ and we consider
$C=\conv{e_1,\ldots,e_{k+1}}\subset \R^{k+1}$, we have
\begin{equation*}
[x_1\cdots x_{k+1}]\conv{e_1,\ldots,e_{k+1}}=\conv{x_1,\ldots,x_{k+1}}.
\end{equation*}
If $\dim{C}$ denotes the dimension of the affine hull of $C$, then
$$\dim{([x_1\cdots x_k]C)} = \min(\mathop{\rm rank}([x_1\cdots x_k]),
\dim{C});$$ moreover, for almost every $x_1,\ldots,x_k\in \R^n$, we
have $\mathop{\rm rank}([x_1\cdots x_k]) = \min(k,n)$.

Let $f_1,\ldots,f_k$ be non-negative bounded, integrable functions on
$\R^n$ such that $\norm{f_i}_1>0$ for each $i=1,\ldots,k$. For a
compact, convex set $C\subset \R^k$ and $p\not=0$, set
\begin{eqnarray}
  \label{eqn:scriptF}
  \lefteqn{\mathcal{F}_{C,p}(f_1,\ldots,f_k)}\nonumber\\ & & =
  \left(\int_{\R^n}\cdots\int_{\R^n}\abs{[x_1\cdots x_k]C}^p
  \prod_{i=1}^k \frac{f_i(x_i)}{\norm{f_i}_1}dx_1\ldots
  dx_k\right)^{1/p}.
\end{eqnarray} 
Here $\abs{\cdot}$ denotes $m$-dimensional Lebesgue measure, where
$m=\min(k,n,\dim{C})$.

The main result from \cite{PaoPiv_probtake} (see Theorem 3.10 and
Section 4.1) is the following theorem.

\begin{theorem}
  \label{thm:PaoPivAdv}
Let $k$ and $n$ be positive integers and $C\subset \R^k$ a compact
convex set. Let $f_1,\ldots,f_k$ be non-negative integrable functions
such that $\norm{f_i}_1>0$ for $i=1,\ldots,k$. Then for each $p\not
=0$,
\begin{equation}
  \label{eqn:Adv1}
   \mathcal{F}_{C,p}(f_1,\ldots,f_k) \gr
   \mathcal{F}_{C,p}(f^*_1,\ldots,f^*_k).
\end{equation}
Moreover, if $\norm{f_i}_{\infty}\ls 1=\norm{f_i}_{1}$ for each
$i=1,\ldots,k$ and $p \gr 1$, then
\begin{equation}
  \label{eqn:Adv2}
   \mathcal{F}_{C,p}(f^*_1,\ldots,f^*_k) \gr
   \mathcal{F}_{C,p}(\mathds{1}_{D_n},\ldots,\mathds{1}_{D_n}).
\end{equation}
\end{theorem}

Under suitable assumptions on $C$, the condition $p\gr 1$ can be
relaxed (e.g., when $\abs{[x_1\cdots x_k]C}^p$ is coordinate-wise
increasing analogous to \cite[Lemma 4.3]{CFPP}).

In \cite{PaoPiv_probtake}, the latter result was stated with the
additional assumption that $k\gr n$ and that $C\subset \R^k$ is a
convex body (so $\dim{C}=k$). In fact, the argument given there works
for any positive integer $k$ and any compact convex set $C\subset
\R^k$.  If $k\ls n$ and $\dim{C}=k$, the matrix $X=[x_1\cdots x_k]$
represents an embedding from $\R^k$ into $\R^n$ and
\begin{equation}
  \label{eqn:simplex_id}
  \abs{[x_1\cdots x_k]C} = \mathop{\rm det}(X^*X)^{1/2}\abs{C};
\end{equation}see, e.g., \cite[Chapter 3]{EG}.
In this case, the quantities $\mathcal{F}^p_{C,p}(f_1,\ldots,f_k)$ are
all multiples of
\begin{equation*}
  \int_{\R^n}\cdots\int_{\R^n} \abs{\conv{0,x_1,\ldots,x_k}}^p
  \prod_{i=1}^k \frac{f_i(x_i)}{\norm{f_i}_1} dx_1\ldots dx_k
\end{equation*}
(cf. (\ref{eqn:conv0})).  When $k>n$, the geometry of $C$ plays a more
significant role, and choosing $C$ suitably gives rise to a number of
isoperimetric inequalities (which was our main interest in
\cite{PaoPiv_probtake}).

It will be useful to have a non-normalized variant of Theorem
\ref{thm:PaoPivAdv} which relaxes the assumption $\norm{f_i}_{\infty}
\ls 1 = \norm{f_i}_1$.  In fact, there are several such variants,
depending on the homogeneity properties of the integrand in
(\ref{eqn:scriptF}).

For subsequent reference, we record two basic identities concerning
the volume of the sets $[x_1\cdots x_k]C$, where $x_1,\ldots,x_k\in
\R^n$, and $C\subset \R^k$ is a compact convex set. Note first that
for each $a>0$,
 \begin{equation}
   \label{eqn:homo_1}
  \abs{[ax_1\cdots a x_k]C} = a^{m} \abs{[x_1\cdots x_k]C},
\end{equation}
 where $m=\min(\mathop{\rm rank}([x_1\cdots x_k]), \dim{C})$.
 Moreover, if $k\ls n$ and $\dim{C} =k$, and $a_1,\ldots,a_k\in
 \R^{+}$, then
\begin{equation}
  \label{eqn:homo_2}
  \abs{[a_1 x_1\cdots a_k x_k]C} = a_1\cdots a_k
  \abs{[x_1\cdots x_k]C},
\end{equation}which follows from (\ref{eqn:simplex_id}).

\begin{corollary}
  \label{cor:B}
  Let $1\ls k \ls n$ and $f_1,\ldots,f_{k}$ be non-negative, bounded
  integrable functions on $\R^n$ such that $\norm{f_i}_1>0$ for each
  $i=1,\ldots,k$.  For $p\in \R$, set
  \begin{equation}
    \label{eqn:cor:B}
  \Delta^0_{p}(f_1,\ldots,f_k) =
  \int_{\R^n}\cdots\int_{\R^n}\abs{\conv{0,x_1,\ldots,x_k}}^p
  \prod_{i=1}^k f_i(x_i)dx_1\ldots dx_k.
\end{equation}
  Then for $p>0$,
  \begin{equation}
    \label{eqn:B}
    \Delta^0_{p}(f_1,\ldots,f_k) \gr
    \left(\prod_{i=1}^k\frac{\norm{f_i}_1^{1+p/n}}
         {\omega_{n}^{1+p/n}\norm{f_i}_{\infty}^{p/n}}\right)
         \Delta^0_{p}(\mathds{1}_{B_2^n},\ldots,\mathds{1}_{B_2^n}).
   \end{equation}
When $-(n-k+1)<p<0$, the inequality is reversed.  Assume additionally
that $\{x: f_i(x)=\norm{f_i}_{\infty}\}$ is a bounded subset of $\R^n$
for $i=1,\ldots,k$ and $p\not =0$. Then equality holds in
(\ref{eqn:B}) for $k=n$ if and only if there is an origin-symmetric
ellipsoid $\mathcal{E}$ and positive constants $a_i$, $b_i$, such that
$f_i=a_i\mathds{1}_{b_i\mathcal{E}}$ a.e. for $i=1,\ldots,k$; for
$k<n$, equality holds in (\ref{eqn:B}) if and only if there are
positive constants $a_i$, $b_i$, such that $f_i=a_i\mathds{1}_{b_i
  B_2^n}$ a.e. for $i=1,\ldots,k$.
\end{corollary}

The condition $-(n-k+1)<p$ is needed for integrability. Since we treat
the equality cases in the latter corollary but there is no discussion
of equality cases in Theorem \ref{thm:PaoPivAdv}, it will be useful to
recall one step in the proof of (\ref{eqn:Adv2}). The basic ingredient
is the next lemma, see e.g., \cite[Lemma 3.5]{PaoPiv_probtake},
\cite[Proof of Lemma 4.3]{CFPP}; the equality condition is not stated
in the latter articles but it is easily obtained from the
proofs. Here, as above, $r_n=\omega_n^{-1/n}$ is the radius of the
Euclidean ball $D_n$ of volume one.

\begin{lemma}
  \label{lemma:bathtub}
  Let $f:\R^{+}\rightarrow [0,1]$ and suppose that
  $\int_{0}^{\infty}f(r)r^{n-1}dr = \int_0^{r_n} r^{n-1}dr$. Then for
  any increasing function $\phi:\R^{+}\rightarrow \R^{+}$, we have
  \begin{equation}
    \int_0^{\infty}\phi(r)f(r)r^{n-1}dr \gr \int_0^{r_n} \phi(r)
    r^{n-1}dr.
  \end{equation}
  If $\phi$ is strictly increasing, then equality holds if and only if
  $f=\mathds{1}_{[0,r_n]}$ a.e.
\end{lemma}

\begin{proof}[Proof of Corollary \ref{cor:B}]For $i=1,\ldots,k$, let 
  $a_i = (\norm{f_i}_1/\norm{f_i}_{\infty})^{1/n}$ and let
\begin{equation}
  \bar{f_i}(x) = \frac{f_i(a_ix)}{\int_{\R^n}f_i(a_iy)dy}.
\end{equation}
Then $\norm{\bar{f_i}}_{1}=\norm{\bar{f_i}}_{\infty}=1$. Using
homogeneity property (\ref{eqn:homo_2}) for the set
$C=\conv{0,e_1,\ldots,e_k}$, we have
\begin{eqnarray}
  \label{eqn:Delta01}
  \Delta_p^0(f_1,\ldots,f_k)& = & 
  \mathcal{F}^p_{C,p}(f_1,\ldots,f_k)\prod_{i=1}^k\norm{f_i}_1 \\
  \label{eqn:Delta02}
  &= & \mathcal{F}^p_{C,p}(\bar{f_1},\ldots,\bar{f_k})\prod_{i=1}^k\frac{\norm{f_i}^{1+p/n}}{\norm{f_i}_{\infty}^{p/n}}.
\end{eqnarray}
Repeating the latter identities for $f_i=\mathds{1}_{B_2^n}$ and
$\bar{f_i}=\mathds{1}_{D_n}$, $i=1,\ldots,k$ and applying Theorem
\ref{thm:PaoPivAdv} gives the desired inequality for $p \gr 1$. Next,
for $p>0$, let $$F(x_1,\ldots,x_k) =
\abs{\conv{0,x_1,\ldots,x_k}}^p.$$ Fix $x_1,\ldots, x_{k-1}\in \R^n$
and $\theta \in S^{n-1}$ such that $x_1,\ldots,x_{k-1}, \theta$ are
linearly independent. Then
  \begin{equation*}
    \R^{+}\ni r\mapsto F(x_1,\ldots,x_{k-1}, r \theta)
  \end{equation*} 
is strictly increasing (cf. (\ref{eqn:homo_2})). For $i=1,\ldots,k$,
let $\bar{f_i}^*=(\bar{f_i})^*$.  One can now extend the result to
$p>0$ by writing $\Delta_p^0(\bar{f_1}^*,\ldots,\bar{f_k}^*)$ in
spherical coordinates and applying Lemma \ref{lemma:bathtub}
iteratively (as in, e.g., \cite[Lemma 4.3]{CFPP}). 

Towards the equality cases, assume that equality holds in
(\ref{eqn:B}). It follows from (\ref{eqn:Delta01}) and
(\ref{eqn:Delta02}) that
\begin{equation}
  \Delta_p^0(\bar{f_1}, \ldots, \bar{f_k}) =
  \Delta_p^0(\mathds{1}_{D_n}, \ldots, \mathds{1}_{D_n}).
\end{equation} 
Furthermore, it follows from Theorem \ref{thm:PaoPivAdv} applied to
$C=\conv{0,e_1,\ldots, e_k}$, that
\begin{equation}
  \label{eqn:Delta0equality}
  \Delta_p^0(\bar{f_1}^*,\ldots,\bar{f_k}^*) =
  \Delta_p^0(\mathds{1}_{D_n},\ldots,\mathds{1}_{D_n}).
\end{equation}

We will first argue that for each $i=1,\ldots,k$, we must have
$\bar{f_i}^* = \mathds{1}_{D_n}$ a.e.  Suppose towards a contradiction
that the latter does not hold.  Without loss of generality, we may
assume that
$$\abs{\{x\in \R^n:\bar{f_k}^*(x)\not = \mathds{1}_{D_n}(x)\}} >0.$$
Then $h(r):=\bar{f_k}^*(r\theta)$ ($r>0$) is independent of $\theta\in
S^{n-1}$ and $h$ differs from $\mathds{1}_{[0,r_n]}$ on a subset of
positive measure.  The equality condition in Lemma \ref{lemma:bathtub}
implies
\begin{equation*}
  \int_{0}^{\infty} F(x_1,\ldots,x_{k-1},r\theta) h(r) r^{n-1}dr >
  \int_{0}^{r_n}F(x_1,\ldots,x_{k-1},r\theta) r^{n-1}dr.
\end{equation*} Integrating in $\theta\in S^{n-1}$,
\begin{eqnarray*}
  \lefteqn{\int_{S^{n-1}}\int_{0}^{\infty}
    F(x_1,\ldots,x_{k-1},r\theta) \bar{f_k}^*(r\theta) r^{n-1} dr
    d\sigma(\theta) }\\ & & >\int_{S^{n-1}}\int_{0}^{r_n}
  F(x_1,\ldots,x_{k-1},r\theta) r^{n-1} dr d\sigma(\theta).
\end{eqnarray*}  
In other words, for linearly independent $x_1,\ldots,x_{k-1}$, we have
\begin{equation*}
  \int_{\R^n} F(x_1,\ldots,x_{k-1},x_k) (\bar{f_k})^*(x_k) dx_k >
  \int_{D_n} F(x_1,\ldots,x_{k-1},x_k) dx_k.
\end{equation*}By continuity of $F$, we have
\begin{equation*}
\Delta_p^0(\bar{f_1}^*, \ldots,\bar{f_k}^*) >
\Delta_p^0(\mathds{1}_{D_n}, \ldots,\mathds{1}_{D_n}),
\end{equation*}which contradicts (\ref{eqn:Delta0equality}).

Thus we have shown that for each $i=1,\ldots,k$,
$\bar{f_i}^*=\mathds{1}_{D_n}$ a.e.  It follows that $\bar{f_i}=
\mathds{1}_{K_i}$, where $K_i$ is a measurable set of volume one, for
$i=1,\ldots, k$.  Since each set $\{f_i=\norm{f_i}_{\infty}\}$ is
bounded, so too is each $K_i$.  Thus we have reduced the equality
cases in (\ref{eqn:B}) to that of bounded, Borel measurable sets and
we appeal to the work of Gardner \cite[Corollary 4.2]{Gardner_dual}
which draws on Pfiefer \cite{Pfiefer}, \cite{Pfiefer90}; we note that
the latter articles state the equality conditions under the assumption
that $K_1=\ldots = K_k$, although it is explained in \cite[pgs
  69-70]{Pfiefer} that the same techniques apply when the bodies $K_i$
are not necessarily the same. 

The case $p<0$ is proved in the same way.
 \end{proof}

Another variant of Theorem \ref{thm:PaoPivAdv} is the following
result.

\begin{theorem}
  \label{thm:A}
  Let $k$ and $n$ be positive integers. Let $f$ be a non-negative
  bounded, integrable function on $\R^n$ with $\norm{f}_1>0$.  Let
  $C\subset \R^k$ be a compact convex set and let $p\gr 1$. Set
  $m=\min(k,n,\dim{C})$. Then  \begin{equation*}
    \mathcal{F}_{C,p}(f,\ldots,f) \gr
    \left(\frac{\norm{f}_1}{\omega_n\norm{f}_{\infty}}\right)^{m/n}
    \mathcal{F}_{C,p}(\mathds{1}_{B_2^n}, \ldots, \mathds{1}_{B_2^n}), 
  \end{equation*}
  where the arguments in $\mathcal{F}_{C,p}(\cdot,\ldots,\cdot)$ are
  repeated $k$-times.
\end{theorem}

\begin{proof}Set  
  $a = (\norm{f}_1/\norm{f}_{\infty})^{1/n}$ and let
\begin{equation}
  \label{eqn:barf}
  \bar{f}(x) = \frac{f(ax)}{\int_{\R^n}f(ay)dy}.
\end{equation}
Then $\norm{\bar{f}}_{1}=\norm{\bar{f}}_{\infty}=1$. Using homogeneity
property (\ref{eqn:homo_1}), we have
\begin{equation*}
  \mathcal{F}_{C,p}(\bar{f},\ldots,\bar{f})=
  \left(\frac{\norm{f}_{\infty}}{\norm{f}_{1}}\right)^{m/n}
  \mathcal{F}_{C,p}(f,\ldots,f).
\end{equation*}
Repeating the latter argument with $f=\mathds{1}_{B_2^n}$ and
$\bar{f}=\mathds{1}_{D_n}$, and applying Theorem \ref{thm:PaoPivAdv}
gives the desired inequality.
\end{proof}

\begin{corollary}
  \label{cor:A}
  Let $1\ls k \ls n$ and let $f$ be a non-negative, bounded integrable
  function on $\R^n$ with $\norm{f}_1>0$.  For $p\not =0$, set
  \begin{equation}
    \Delta_p(f,\ldots,f) =
    \int_{\R^n}\cdots\int_{\R^n}\abs{\conv{x_1,\ldots,x_{k+1}}}^p
    \prod_{i=1}^{k+1} f(x_i)dx_1\ldots dx_{k+1}.
  \end{equation}
   Then for $p\gr 1$, 
  \begin{equation}
    \label{eqn:cor:A}
    \Delta_p(f,\ldots,f) \gr
    \frac{\norm{f}_1^{k+1+kp/n}}{\omega_{n}^{k+1+kp/n}\norm{f}_{\infty}^{kp/n}}
    \Delta_{p}(\mathds{1}_{B_2^n},\ldots,\mathds{1}_{B_2^n}).
   \end{equation}
Assume additionally that $\{f=\norm{f}_{\infty}\}$ is a bounded subset
of $\R^n$. Then equality holds in (\ref{eqn:cor:A}) when $k=n$ if and
only if there is an ellipsoid $\mathcal{E}$ and a positive constant
$a$ such that $f=a\mathds{1}_{\mathcal{E}}$ a.e.; when $k<n$, equality
holds if and only if there is a positive constant $a$ and a Euclidean
ball $B$ such that $f=a\mathds{1}_{B}$ a.e.
\end{corollary}

The proof is parallel to that of Corollary \ref{cor:B}, although the
equality conditions in this case require the following additional
lemma (see \cite[Lemmas 3.7, 3.8]{PaoPiv_probtake}).
\begin{lemma}
  \label{lemma:avg_inc}
  Let $1\ls k\ls n$ and let $x_1,\ldots,x_{k}\in \R^n$.
  Then for each $p\gr 1$,
  \begin{equation}
    \R^{+}\ni r\mapsto
    \int_{S^{n-1}}\abs{\conv{x_1,\ldots,x_{k},r\theta}}^p d\sigma(\theta)
  \end{equation}is increasing.
\end{lemma}

Here and throughout, ``increasing'' is used in the non-strict sense.

\begin{remark}If $k=n=1$, the condition $p\gr 1$ in the latter lemma is needed.
Indeed, in this case $S^{0}=\{-1,1\}$ and the function $\R^{+}\ni
r\mapsto \frac{1}{2}(\abs{r-x_1}^p + \abs{r+x_1}^p)$ is not monotone
if $p<1$, $p\not=0$ and $x_1\not =0$.
\end{remark}

\begin{proof}[Proof of Corollary \ref{cor:A}]
  Set $C = \conv{e_1,\ldots,e_{k+1}}\subset \R^{k+1}$ so that
  $\dim{C}=k$.  Observe that
  \begin{equation*}
    \Delta_p(f,\ldots,f) = \norm{f}_1^{k+1}
    \mathcal{F}_{C,p}(f,\ldots,f)^p,
  \end{equation*}
  where the arguments in $\mathcal{F}_{C,p}(\cdot,\ldots,\cdot)$ are
  repeated $k+1$ times.  The inequality follows from Theorem
  \ref{thm:A} (with $k+1$ in place of $k$ and $m=k$).
  
  Assume now that equality holds in (\ref{eqn:cor:A}). It follows that
  \begin{equation}
    \Delta_{p}(\bar{f},\ldots,\bar{f}) =
    \Delta_{p}(\mathds{1}_{D_n},\ldots,\mathds{1}_{D_n}),
  \end{equation}where $\bar{f}$ is defined in (\ref{eqn:barf}).
  In turn, we must have equality in both inequalities in Theorem
  \ref{thm:PaoPivAdv}. In particular,
  \begin{equation}
    \label{eqn:Deltaequality}
    \Delta_{p}(\bar{f}^*,\ldots,\bar{f}^*) =
    \Delta_{p}(\mathds{1}_{D_n},\ldots,\mathds{1}_{D_n}) \, ,
  \end{equation}  where $ \bar{f}^*:= (\bar{f})^*$.
  As above, we claim that $\bar{f}^* = \mathds{1}_{D_n}$ a.e.  For a
  contradiction, we assume that
  $$\abs{\{x\in \R^n:\bar{f}^*(x)\not 
    = \mathds{1}_{D_n}(x)\}} >0.$$ 
  
  Let $$F(x_1,\ldots,x_{k+1}) = \abs{\conv{x_1,\ldots, x_{k+1}}}^p.$$
  By Lemma \ref{lemma:avg_inc}, for any $x_1,\ldots, x_{k}\in
  \R^n$, the function
  \begin{equation*}
    \R^{+}\ni r\mapsto \int_{S^{n-1}}F(x_1,\ldots,x_{k}, r \theta)
    d\sigma(\theta)
  \end{equation*} is increasing.  
  By Lemma \ref{lemma:bathtub}, we have
  \begin{eqnarray*}
    \lefteqn{\int_{0}^{\infty} \int_{S^{n-1}}F(x_1,\ldots,x_{k},r\theta)
    \bar{f}^*(r\theta) r^{n-1}d\sigma(\theta)dr}\\
    & & \gr
    \int_{0}^{r_n}\int_{S^{n-1}}F(x_1,\ldots,x_{k},r\theta)
    r^{n-1}d\sigma(\theta)dr. 
  \end{eqnarray*}  
  i.e.
  \begin{equation}
    \label{eqn:1d_bath}
    \int_{\R^n} F(x_1,\ldots,x_{k},x_{k+1}) \bar{f}^*(x_{k+1})
    dx_{k+1} \gr \int_{D_n} F(x_1,\ldots,x_{k},x_{k+1}) dx_{k+1}.
  \end{equation}
  Assume now that $x_1,\ldots,x_k$ are affinely independent points
  inside the support of $\bar{f}^*$, and $0\in \conv{x_1,\ldots,x_k}$.
  Then \begin{equation*} F(x_1,\ldots,x_k,r\theta)^{1/p}=
    (1/k)\abs{\conv{x_1,\ldots,x_k}} \abs{P_E (r\theta)},
    \end{equation*}where  $E =\mathop{\rm span}\{x_1,\ldots,x_k\}^{\perp}$.  In particular, for
  such fixed $x_1,\ldots,x_k$, $F$ is a strictly increasing function
  of $r$. Consequenlty, for such $x_i$ the inequality in
  (\ref{eqn:1d_bath}) is strict by Lemma \ref{lemma:bathtub}.  By
  continuity of $F$, and another application of Theorem \ref{thm:A},
  we get
  \begin{equation*}
    \Delta_p(\bar{f}^*,\ldots,\bar{f}^*, \bar{f}^*) >
    \Delta_p(\bar{f}^*,\ldots,\bar{f}^*, \mathds{1}_{D_n})
    \gr \Delta_p(\mathds{1}_{D_n}, \ldots,\mathds{1}_{D_n}),
  \end{equation*}
  which contradicts (\ref{eqn:Deltaequality}).
  
  Thus $\bar{f}^*=\mathds{1}_{D_n}$ a.e., hence
  $\bar{f}=\mathds{1}_K$, for some measurable set $K$ of volume
  one. Arguing as in the proof of Corollary \ref{cor:B}, we reduce the
  equality case in (\ref{eqn:cor:A}) to that of bounded Borel
  measurable sets and we appeal again to the work of Gardner
  \cite[Corollary 4.2]{Gardner_dual} and Pfiefer \cite{Pfiefer},
  \cite{Pfiefer90}.
\end{proof}


\section{Integral inequalities on $G_{n,k}$ and $M_{n,k}$}
\label{section:main_results}

In this section we prove Theorems \ref{thm:Busemann} and
\ref{thm:Schneider}. We start with a generalization of the former.

\begin{theorem}
  \label{thm:Grinberg_general}
Let $1\ls q\ls k\ls n-1$ and let $f_1,\ldots,f_{q}$ be non-negative,
bounded integrable functions on $\R^n$.  Then for $0\ls p\ls n-k$,
\begin{equation}
  \label{eqn:Grinberg_general}
  \int_{G_{n,k}} \prod_{i=1}^{q}
  \frac{\norm{f_i\vert_{E}}_1^{1+p/k}}{\norm{f_i\vert_{E}}_{\infty}^{p/k}}
  dE \ls \frac{\omega_{k}^{q(k+p)/k}}{\omega_n^{q(k+p)/n}}
  \prod_{i=1}^q \norm{f_i}_1^{(k+p)/n}\norm{f_i}_{\infty}^{(n-k-p)/n}.
\end{equation}
In particular, when $q=k$ and $p=n-k$, we have
\begin{equation}
  \label{eqn:Grinberg_general_2}
  \int_{G_{n,k}} \prod_{i=1}^k
  \frac{\norm{f_i\vert_E}_1^{n/k}}{\norm{f_i\vert_E}_{\infty}^{(n-k)/k}}dE
    \ls \frac{\omega_k^n}{\omega_n^k}\prod_{i=1}^k \norm{f_i}_1.
\end{equation}
Assume additionally that $\{f_i=\norm{f_i}_{\infty}\}$ is bounded and
$f_i$ is continuous at $0$ for each $i=1,\ldots,k$ and $p>0$. Then
equality holds in (\ref{eqn:Grinberg_general}) for $k=1$ if and only
if $f_1=a\mathds{1}_K$ a.e., where $a>0$ and $K\subset \R^n$ is
star-shaped about the origin; for $1\ls q<k$ or $0<p<n-k$ if and only
if there are positive constants $a_i$, $b_i$, such that
$f_i=a_i\mathds{1}_{b_i B_2^n}$ a.e.; equality holds in
(\ref{eqn:Grinberg_general_2}) if and only if there is an
origin-symmetric ellipsoid $\mathcal{E}\subset \R^n$ and positive
constants $a_i$, $b_i$, such that $f_i=a_i\mathds{1}_{b_i\mathcal{E}}$
a.e. for $i=1,\ldots,k$.
\end{theorem}

\begin{proof}[Proof of Theorem \ref{thm:Grinberg_general}]
 We will use the following well-known identity from integral geometry:
 for $p>-(k-q+1)$, we have
  \begin{equation}
    \label{eqn:Delta_neg_Gnk}
    \Delta^0_{-(n-k-p)}(f_1,\ldots,f_q) =
    c_{n,k,q}\int_{G_{n,k}}\Delta_{p}^0(f_1\vert_E,\ldots,f_q\vert_E)dE,
  \end{equation} 
  where $\Delta_p^0$ is defined in (\ref{eqn:cor:B}); this is simply
  Theorem \ref{thm:BPGnk} applied to the function $$F(x_1,\ldots,x_q)
  = \prod_{i=1}^q f_i(x_i) \abs{\mathop{\rm
      conv}\{0,x_1,\ldots,x_q\}}^{-(n-k-p)}.$$
  
  Assume now that $E\in G_{n,k}$ and $\norm{f_i\vert_E}_{\infty} >0$
  for $i=1,\ldots,q$.  Applying Corollary \ref{cor:B} on $E$ (with $k$
  in place of $n$ and $q$ in place of $k$), we have
  \begin{equation}
    \label{eqn:pointwise_Gnk}
    \prod_{i=1}^{q} \frac{\norm{f_i\vert_E}_1^{1+p/k}}
         {\norm{f_i\vert_{E}}_{\infty}^{p/k}} \ls
         \omega_{k}^{q(k+p)/k}\frac{\Delta^0_{p}(f_1\vert_{E},\ldots,f_{q}\vert_{E})}
               {\Delta^0_{p}(\mathds{1}_{B_2^n\cap
                   E},\ldots,\mathds{1}_{B_2^n\cap E})}.
  \end{equation}
  Note that $\Delta^0_{p}(\mathds{1}_{B_2^n\cap
    E},\ldots,\mathds{1}_{B_2^n\cap E})$ is independent of $E$.  Thus,
  integrating over $G_{n,k}$, using (\ref{eqn:Delta_neg_Gnk}) and Corollary
  \ref{cor:B} once more, we have
  \begin{eqnarray}
    \int_{G_{n,k}} \prod_{i=1}^{q} \frac{\norm{f_i\vert_E}_1^{1+p/k}}
        {\norm{f_i\vert_{E}}_{\infty}^{p/k}} dE &\ls &
        \omega_{k}^{q(k+p)/k}
        \frac{\int_{G_{n,k}}\Delta^0_{p}(f_1\vert_{E},\ldots,f_{q}\vert_{E})dE}
             {\int_{G_{n,k}}\Delta^0_{p}(\mathds{1}_{B_2^n\cap
                 E},\ldots,\mathds{1}_{B_2^n\cap E})dE}\label{eqn:Gnk_a}\\ & = &
             \omega_{k}^{q(k+p)/k}
             \frac{\Delta_{-(n-k-p)}^0(f_1,\ldots,f_q)}
                  {\Delta_{-(n-k-p)}^0(\mathds{1}_{B_2^n},
               \ldots, \mathds{1}_{B_2^n})} \\ & \ls &
             \frac{\omega_k^{q(k+p)/k}}{\omega_n^{q(k+p)/n}}
             \prod_{i=1}^q
             \norm{f_i}_1^{(k+p)/n}\norm{f_i}_{\infty}^{(n-k-p)/n}.
             \label{eqn:Gnk_b}
  \end{eqnarray}
The equality cases follow from those of Corollary \ref{cor:B} and
Gardner's characterizations of sets that are ellipsoids, Euclidean
balls or star-shaped, up to sets of measure zero, in \cite[Section
  6]{Gardner_dual}.
\end{proof}

\begin{theorem}
  \label{thm:Schneider_general}
  Let $1\ls k \ls n-1$ and $f$ be a non-negative bounded integrable
  function on $\R^n$. Then
  \begin{equation}
    \label{eqn:Schneider_general}
    \int_{M_{n,k}} \frac{\left(\int_{F}f(x)dx\right)^{n+1}}
        {\norm{f\vert_{F}}_{\infty}^{n-k}}dF \ls
        \frac{\omega_k^{n+1}\omega_{n(k+1)}}{\omega_n^{k+1}\omega_{k(n+1)}}
        \left(\int_{\R^n}f(x)dx\right)^{k+1}.
  \end{equation}
  Assume additionally that $\{f=\norm{f}_{\infty}\}$ is bounded. Then
  equality holds in (\ref{eqn:Schneider_general}) when $k=1$ if and
  only if $f=\mathds{1}_K$, where $K\subset \R^n$ is a convex body;
  when $k>1$ if and only if there a positive constant $a$ and an
  ellipsoid $\mathcal{E}$ such that $f=a\mathds{1}_{\mathcal{E}}$ a.e.
\end{theorem}

When $f=\mathds{1}_K$, where $K$ is a convex body in $\R^n$, a more
general result is due to Schneider \cite{Schneider_flats}: for $s\in
\{1,\ldots,n\}$,
\begin{equation}
  \label{eqn:Schneider_p}
  \int_{M_{n,k}}\abs{K\cap F}^{s+1}dF \ls
  \frac{\omega_k^{s+1}}{\omega_n^{(n+ks)/n}}\frac{\omega_{n+ks}}{\omega_{k+ks}}\abs{K}^{1+ks/n}.
\end{equation} 
It is natural to try to extend Theorem \ref{thm:Schneider_general} to
powers $1\ls s\ls n$ as in (\ref{eqn:Schneider_p}). Following the line
of proof of Theorem \ref{thm:Grinberg_general} would require a
statement such as Corollary \ref{cor:A} for $p<1$ (in particular
$p<0$); it is unclear to us if this is possible.

\begin{proof}[Proof of Theorem \ref{thm:Schneider_general}]
  Let $F\in M_{n,k}$ and assume $\norm{f\lvert_F}_{\infty}>0$.
  Applying Corollary \ref{cor:A} on $F$ with $p=n-k$ (replacing $n$ by
  $k$) we have
\begin{equation}
  \label{eqn:Delta}
  \frac{\left(\int_{F}f(x)dx\right)^{n+1}}{\norm{f\vert_{F}}_{\infty}^{n-k}}
  \ls \omega_{k}^{n+1}\frac{ \Delta_{n-k}(f|_F,\ldots,f|_F) }
      {\Delta_{n-k}(\mathds{1}_{B_2^{k}},\ldots,\mathds{1}_{B_2^{k}})},
\end{equation} where the arguments in $\Delta_{n-k}$ are repeated $k+1$ times.
Integrating over $M_{n,k}$ and applying Theorem \ref{thm:BPMnk}, we
get
\begin{eqnarray}
  \int_{M_{n,k}}\frac{\left(\int_{F}f(x)dx\right)^{n+1}}{\norm{f\vert_{F}}_{\infty}^{n-k}}dF
  & \ls & \frac{\omega_{k}^{n+1}\int_{M_{n,k}}
    \Delta_{n-k}(f_1|_F,\ldots,f_{k+1}|_F) dF}
      {\Delta_{n-k}(\mathds{1}_{B_2^{k}},\ldots,\mathds{1}_{B_2^{k}})}\label{eqn:Delta1}
      \\ & = &
      \frac{\omega_{k}^{n+1}}{c_{n,k,k}\Delta_{n-k}(\mathds{1}_{B_2^{k}},\ldots,\mathds{1}_{B_2^{k}})}
      \left(\int_{\R^n}f(x)dx\right)^{k+1}.\nonumber
\end{eqnarray} 
If $f=\mathds{1}_{B_2^n}$, then inequality (\ref{eqn:Delta}) is an
equality (as noted in \cite{Schneider_flats}), hence so is
(\ref{eqn:Delta1}).  Consequently, using the expression in the
equality case in (\ref{eqn:Schneider}) and rearranging terms, we get
\begin{equation*}
\Delta_{n-k}(\mathds{1}_{B_2^k},\ldots,\mathds{1}_{B_2^k}) = 
c_{n,k,k} \omega_n^{k+1} \frac{\omega_{k(n+1)}}{\omega_{n(k+1)}},
\end{equation*} The latter also follows from results of Kingman \cite{Kingman} and Miles
\cite{Miles}.  This proves the inequality.  

The equality cases follow from those of Corollary \ref{cor:A} and
\cite[Corollary 6.8]{Gardner_dual}.
\end{proof}


\section{Bounds for marginals}
\label{section:marginals}

In this section we state and prove a generalization of Theorem
\ref{thm:marginal_intro}. We also discuss marginals of log-concave
measures and connections to the Hyperplane Conjecture.

\begin{theorem} 
  \label{thm:marginal_body}
  Suppose $\mu$ is a probability measure on $\R^n$ with a bounded
  density $f$. Then for each $1\ls k \ls n-1$ and $s>1$, there exists
  ${\cal{A}}_{s}\subseteq G_{n,k}$ with $\mu_{n,k}( {\cal{A}}_{s}) \gr
  1- 2s^{-kn}$ such that:
  \begin{itemize}
  \item[(i)] for every $E\in {\cal{A}}_{s}$ and $t>1$, there exists a
    set ${\cal{B}}_{t} \subseteq E$ such that $\pi_{E}(\mu)( {\cal{B}}_{t})
    \ls t^{-kn}$ and
    \begin{equation}
      \label{eqn:marginal_body_1}
      f_{\pi_E(\mu)}(x)^{1/k}\ls c_{1}st\norm{f}_{\infty}^{1/n}, \quad ( x\in
      (E\setminus {\cal{B}}_{t}) \cup\{0\});
    \end{equation} 
  \item[(ii)] for every $E\in {\cal{A}}_{s}$, $\varepsilon>0$ and any $z\in E$, 
    \begin{equation}
      \label{eqn:marginal_body_2}
      \pi_{E}(\mu) \left(\{x\in E : \abs{x-z} \ls \varepsilon \sqrt{k}
        \}\right) \ls (c_{2}s\varepsilon \norm{f}_{\infty}^{1/n})^{kn/(n+1)}.
    \end{equation}
  \end{itemize}
\end{theorem}

\begin{proof}
By Fubini's Theorem and Theorem \ref{thm:Schneider}, 
\begin{eqnarray*}
\left(\int_{G_{n,k}} \int_{E} \frac{\left( \int_{E^{\perp} +x} f(y )
    d y \right)^{n}}{ \norm{f\vert_{E^{\perp}+ x}}_{\infty}^{k}} d
  \pi_{E}(\mu)(x) dE \right)^{\frac{1}{kn}}
    & = &
  \left(\int_{M_{n,n-k}} \frac{ \left( \int_{F} f(y ) d y \right)^{n+1}}{
    \norm{f\vert_{F}}_{\infty}^{k}} dF\right)^{\frac{1}{kn}} \\ & \ls &
  \left(\frac{\omega_{n-k}^{n+1} \omega_{n(n-k+1)}}{\omega_{n}^{n-k+1}
    \omega_{(n-k)(n+1)}}\right)^{\frac{1}{kn}} \\
  & \simeq & 1.
\end{eqnarray*}
By Markov's inequality, for each $s>1$, the $\mu_{n,k}$-measure of the
set ${\cal{A}}_{s}^{(1)}$ of $E\in G_{n,k}$ such that
\begin{equation}
  \label{eqn:Markov_E}
\int_{E} \frac{ \left(\int_{E^{\perp} +x} f(y) dy \right)^{n}}{
  \norm{f\vert_{E^{\perp}+x}}_{\infty}^{k}} d \pi_{E}(\mu)(x) \ls (c
s)^{kn}
\end{equation} 
is at least $1-s^{-kn}$.  For $E\in {\cal{A}}_{s}^{(1)}$ and $t>1$, we
apply Markov's inequality once more to get that the
$\pi_E(\mu)$-measure of the set ${\cal{B}}_{t}$ of $x\in E$ such that
\begin{equation}
  \label{eqn:Markov_x}
\frac{ \left( \int_{E^{\perp} +x} f(y) dy \right)^{n}}{
  \norm{f\vert_{E^{\perp}+ x}}_{\infty}^{k}} \gr (cst)^{kn}
\end{equation}
is less than $t^{-kn}$. Note that for every $x\in E\setminus
{\cal{B}}_{t}$,
\begin{equation*}
\left(\int_{E^{\perp}+x} f(y) dy \right)^{1/k}
\ls  cst \norm{f\vert_{E^{\perp}+x}}_{\infty}^{1/n}  
\ls cst  \norm{f}_{\infty}^{1/n}.
\end{equation*}

Applying now Theorem \ref{thm:Busemann}, we have
\begin{equation*}
  \left( \int_{G_{n,k}} \frac{ f_{\pi_{E}(\mu)}(0)^{n }}{
  \norm{f\vert_{E^{\perp}}}_{\infty}^{k}} dE \right)^{\frac{1}{nk}} = \left(
\int_{G_{n,k}} \frac{ \left(\int_{E^{\perp}} f(y) d y \right)^{n}}{
  \norm{f\vert_{E^{\perp}}}_{\infty}^{k}} dE \right)^{\frac{1}{nk}} \ls
\left( \frac{ \omega_{n-k}^{n}}{ \omega_{n}^{n-k}}
\right)^{\frac{1}{nk}} \simeq 1. 
\end{equation*}
By Markov's inequality, for each $s>1$, the $\mu_{n,k}$-measure of the
set ${\cal{A}}_{s}^{(2)}$ of $E\in G_{n,k}$ such that
\begin{equation*}
f_{\pi_{E}(\mu)}(0)^{\frac{1}{k}} \ls cs
\norm{f\vert_{E^{\perp}}}_{\infty}^{\frac{1}{n}} 
\end{equation*}
is at least $1-s^{-nk}$.  For every $E\in {\cal{A}}_{s}^{(2)}$, we have
that
\begin{equation*}
  f_{\pi_{E}(\mu)}(0)^{\frac{1}{k}} \ls cs
  \norm{f\vert_{E^{\perp}}}_{\infty}^{\frac{1}{n}} \ls cs
  \norm{f}_{\infty}^{\frac{1}{n}}.
\end{equation*}
Setting ${\cal{A}}_{s}:= {\cal{A}}_{s}^{1} \cup {\cal{A}}_{s}^{(2)}$,
we conlude that (\ref{eqn:marginal_body_1}) holds.

Towards (\ref{eqn:marginal_body_2}), let $s,t>1$, $E\in
{\cal{A}}_{s}$, $z\in E$ and $\eps>0$. Then
\begin{eqnarray*}
\lefteqn{\pi_{E}(\mu) \left( \{x\in E : \abs{x-z} \ls \varepsilon \sqrt{k}
\}\right)}\\
& & =  \int_{B(z, \varepsilon \sqrt{k})} f_{\pi_{E}(\mu)}(y) dy\\
& & = \int_{B(z, \varepsilon \sqrt{k})\cap {\cal{B}}_{t}} f_{\pi_{E}(\mu)}(y) d y 
+ \int_{B(z, \varepsilon \sqrt{k}) \cap {\cal{B}}_{t}^{c}} f_{\pi_{E}(\mu)}(y) d y \\
& & = \pi_{E}(\mu) \left( {\cal{B}}_{t}\right) 
+ \sup_{y\in E\setminus {\cal{B}}_{t}} f_{\pi_{E}(\mu)}(y) 
\abs{ B(z, \varepsilon \sqrt{k})} \\
& & \ls  t^{-nk} +( c_{1}st\norm{f}_{\infty}^{1/n})^{k} (c_{0} \varepsilon )^{k} \,. \\
\end{eqnarray*}
We choose $t:= (c_{2} \varepsilon
s\norm{f}_{\infty}^{1/n})^{-\frac{1}{n+1}}$ and we get
(\ref{eqn:marginal_body_2}).
\end{proof}

\begin{remark}
  \label{remark:smallball_exp}
  If ${\cal{A}}_{s}$ is the set in Theorem \ref{thm:marginal_intro},
  $E\in {\cal{A}}_{s}$ and $\varepsilon>0$, then we have the stronger
  small-ball probability
  \begin{equation}
    \pi_{E}(\mu) \left(\{x\in E : \abs{x} \ls \varepsilon \sqrt{k}
    \}\right) \ls (c_{3}s\varepsilon \norm{f}_{\infty}^{1/n})^{k}.  
  \end{equation}
  The proof is analogous to that of (\ref{eqn:marginal_body_2}) (use
  Theorem \ref{thm:Busemann}).
\end{remark}

The next lemma shows that the probability estimate for the
$\mu_{n,k}$-measure in Theorem \ref{thm:marginal_body} is sharp in
each dimension $k$.

\begin{lemma}
  \label{lemma:Gaussian_sharp}
Let $1\ls k \ls n-1$ and set $\sigma=(2\pi)^{-n/(2k)}$. Let $f$ be the
Gaussian density with law $\mu=N(0, D)$, where $D$ is the diagonal
matrix $D= \mathop{\rm diag}(\sigma^2,\ldots,\sigma^2,1,\ldots,1)$,
with $\sigma^2$ repeated $k$ times. Then for each $1\ls s \ls
\sigma^{-1}$,
\begin{equation}
  \label{eqn:Gaussian_sharp}
  \mu_{n,k} \left(\{E\in G_{n,k}:
  \norm{f_{\pi_E(\mu)}}_{\infty}^{1/k} \gr s \}\right) \gr
  (2s)^{-k(n-k)}.
\end{equation}
\end{lemma}

The proof relies on the following proposition, which is a direct
consequence of a result of Szarek \cite{Szarek}; this formulation is
from \cite[Corollary 2.2]{PaourisValettas}; here we equip $G_{n,k}$
with the metric $d(E_0, E_1)$ which is the operator norm
$\norm{P_{E_0}-P_{E_1}:\ell_2^n\rightarrow \ell_2^n}$.

\begin{proposition}
  \label{prop:Szarek_cap}
  Let $1\ls k \ls n-1$, $E\in G_{n,k}$ and $\varepsilon\in (0,2)$. Then
  \begin{equation*}
    \mu_{n,k} \left(\{F\in G_{n,k}: d(E,F) \ls \varepsilon\}\right) \gr
    (c\varepsilon)^{k(n-k)},
  \end{equation*}where $c$ is a numeric constant.
\end{proposition}

\begin{proof}[Proof of Lemma \ref{lemma:Gaussian_sharp}]
Let $g$ be a standard Gaussian vector in $\R^n$.  Let $E_0$ be the
span of first $k$ coordinate unit vectors and note that
    \begin{equation*}
      \norm{f}_{\infty} = (2\pi \sigma^2)^{-k/2} (2\pi)^{-(n-k)/2} = 1 \, ,
  \end{equation*}while
    \begin{equation*}
      \norm{f_{\pi_{E_0}(\mu)}}_{\infty} = (2\pi)^{(n-k)/2}.
    \end{equation*}
Let $\eps>0$ and assume that $E_1\in G_{n,k}$ satisfies $d(E_0, E_1)<
\eps$.  Write
\begin{equation*}
  P_{E_1}D = \sigma P_{E_1} P_{E_0} + P_{E_1}P_{E_0^{\perp}}.  
\end{equation*}
Using singular value decomposition, there exist orthonormal bases
$u_1,\ldots, u_k$ of $E_0$ and $v_1,\ldots,v_k$ of $E_1$ such that
\begin{equation*}
  P_{E_1}P_{E_0} =\sum_{i=1}^k a_i u_i\otimes v_i,
\end{equation*}
where $0\ls a_i=\langle u_i, v_i\rangle\ls 1$.  Since $P_{E_1}
P_{E_0^{\perp}} = P_{E_1}(I-P_{E_0})$, we can write
\begin{eqnarray*}
  P_{E_1}P_{E_0^{\perp}}=\sum_{i=1}^k v_i \otimes v_i - \sum_{i=1}^k
  a_i u_i \otimes v_i =  \sum_{i:a_i\not= 1} \sqrt{1-a_i^2}f_i \otimes v_i,
\end{eqnarray*}where 
$f_i=\frac{v_i-a_i u_i}{\sqrt{1-a_i^2}} $.  Set $\gamma_i=\langle g,
u_i \rangle$ and $\gamma_i'= \langle g, f_i\rangle$. Since $u_i$ and
$f_i$ are orthogonal, $\gamma_i$ and $\gamma_i'$ are independent.
Note that
\begin{equation*}
 P_{E_1}Dg = \sum_{i:a_i=1} \sigma \gamma_i v_i +\sum_{i:a_i\not = 1}
 \left(\sigma a_i \gamma_i +\sqrt{1-a_i^2}\gamma_i'\right)v_i.
\end{equation*} 
Sicne $\sqrt{1-a_i^2} \ls \norm{P_{E_0}-P_{E_1}} = d(E_0,E_1) \ls
\eps$ for each $i=1,\ldots,k$, the covariance matrix $A$ of
$P_{E_1}Dg$ satisfies
\begin{equation*}
(\mathop{\rm det}(A))^{1/k} \ls \frac{1}{k}\mathop{\rm
    tr}(A)=\frac{1}{k} \sum_{i=1}^k (\sigma^2 a_i^2 +1-a_i^2) \ls
  \sigma^2 + \eps^2.
\end{equation*}
It follows that for any $E_1$ with $d(E_0,E_1)\ls \eps$, we have
\begin{equation*}
\norm{f_{\pi_{E_1}(\mu)}}_{\infty}^{1/k} =
\frac{1}{((2\pi)^k\mathop{\rm det}(A))^{1/(2k)}}
\gr \frac{1}{\sqrt{2\pi(\sigma^2+\eps^2)}}.
\end{equation*}
We now apply Proposition \ref{prop:Szarek_cap} with $\eps=1/(cs)$ to
obtain (\ref{eqn:Gaussian_sharp}).
\end{proof}

\subsection{Concluding remarks}

In light of Theorem \ref{thm:marginal_body}, a natural question arises
here: under what additional condition(s), can one guarantee that {\it
  all} marginal densities of such functions $f$ are suitably bounded,
i.e.,
\begin{equation}
\label{result-wish}
\norm{f_{\pi_{E}(\mu)}}_{\infty}^{\frac{1}{k}} \ls C
\norm{f}_{\infty}^{\frac{1}{n}}, \quad \forall E \in G_{n,k},
\end{equation}
 as in the case of Rudelson and Vershynin's result
 \eqref{eqn:RV_bound}. The example of independent Gaussians with
 different variances in Lemma \ref{lemma:Gaussian_sharp} shows that
 one needs a type of non-degeneracy condition. For instance, one may
 assume that $\mu$ is isotropic, namely,
\begin{equation}
\label{assumption-f-1}
 \int_{\mathbb R^{n}} \langle x, \theta \rangle d\mu(x)=0 \text{ and }
 \ \int_{\mathbb R^{n}} |\langle x, \theta \rangle|^{2} d\mu(x)=1
 \quad\forall \theta \in S^{n-1}.
\end{equation}
If $\mu$ is an isotropic, subgaussian, $\log$-concave probability
measure then \eqref{result-wish} holds as a consequence of a result of
Bourgain on the isotropic constant of such measures
\cite{Bourgain_PSI2}.  If $\mu$ is isotropic and $\log$-concave, the
isotropic constant of $\mu$ is defined by $L_{\mu}:=
\norm{f_{\mu}}_{\infty}^{1/n}$. A major open problem known as the
Hyperplane Conjecture asks if there exists an absolute constant $C$
(independent of $n$ and $\mu$) such that $L_{\mu} \ls C$. The best
known bound (of order $n^{1/4}$) is due to Klartag \cite{Klartag_LK},
improving an earlier result of Bourgain \cite{Bourgain_LK}. For
detailed discussion on this conjecture, see \cite{BGVV}. Thus in the
class of isotropic $\log$-concave probability measures $\mu$,
inequality \eqref{result-wish} amounts to asking if
\begin{equation}
\label{result-wish-iso}
 L_{\pi_{E}(\mu)} \ls C L_{\mu} , \ \forall E \in G_{n,k}.
\end{equation}
It is not difficult to show that the above question is just another
equivalent formulation of the Hyperplane Conjecture. (For a proof of
this fact see \cite{Paouris_marginals}). The inequality of
  Busemann-Straus and Grinberg (\ref{eqn:Grinberg}) has been used
  recently in \cite{PaourisValettas} to show that the marginals that
  satisfy the Hyperplane Conjecture form a $1$-net in $G_{n,k}$ for
  $k\ls \sqrt{n}$. One of the main ingredients in the proof is entropy
  numbers on the Grassmanian established by Szarek (Proposition
  \ref{prop:Szarek_cap} above). Using these estimates along with
  Theorem \ref{thm:marginal_intro}, we get the following corollary.

\begin{corollary}
\label{result-corollary}
Let $\mu$ be a probability on $\R^n$ with a bounded density $f$.  Then
for every $1\ls k \ls n-1$, $E\in G_{n,k}$ and $\eta>0$, there exists
$E_0\in G_{n,k}$ with $d(E_0,E)\ls \eta$ such that for any $z\in E_0$,
\begin{equation}
\label{result-4}
 \pi_{E_0}(\mu) \left(\{x\in E_0 : \abs{x-z} \ls \varepsilon \sqrt{k}\}
 \right) \ls
 \left(c_{2}\frac{\varepsilon}{\eta}\norm{f}_{\infty}^{1/n}\right)^{\frac{kn}{n+1}}.
 \end{equation}
\end{corollary}
In other words, given any $E\in G_{n,k}$, there exists $E_0$, close to
$E$ such that $E_0$ has a nearly optimal small-ball probability
estimate.\\

\noindent{\bf Acknowledgements}

It is our pleasure to thank Alex Koldobsky, Fedja Nazarov, Gestur
Olafsson and Petros Valettas for helpful discussions. The first-named
author thanks the Oberwolfach Research Institute for Mathematics for
its hospitality and support, where part of this work was carried out.
The second-named author acknowledges with thanks support from the
A. Sloan foundation, BSF grant 2010288 and the US NSF grant
CAREER-1151711.  The third-named author would like to thank the
University of Missouri Research Board and Simons Foundation (grant
\#317733) for financial support which facilitated our collaboration.

\small \bibliographystyle{amsplain} \bibliography{DanPaoPivbib}

\medskip
\large
\noindent 
Susanna Dann: Vienna University of Technology, Wiedner Hauptstrasse
8-10, 1040 Vienna, Austria\\
{\tt susanna.dann@tuwien.ac.at}
\medskip

\noindent 
Grigoris Paouris: Department of Mathematics, Mailstop 3368, Texas A\&M
University, College Station, TX 77843-3368, USA\\ {\tt
  grigoris@math.tamu.edu}

\medskip

\noindent 
Peter Pivovarov: Mathematics Department, University of Missouri,
Columbia, MO 65211, USA, \\{\tt pivovarovp@missouri.edu}

\end{document}